\documentclass{compositio}
\usepackage{graphicx}
\usepackage{xy}
\usepackage{amsmath}
\usepackage{amsthm}
\usepackage{amsfonts}
\usepackage{amssymb}
\usepackage{cite}
\usepackage{hyperref}
\xyoption{all}
\usepackage{setspace}

\newtheorem{thm}{Theorem}[section]  

\theoremstyle{definition}    
\newtheorem{lemma}[thm]{Lemma}

\newtheorem{lem}[thm]{Lemma}
\newtheorem{cor}[thm]{Corollary}
\newtheorem{rem}[thm]{Remark}

\newcommand{\mb}[1]{{\mathbf #1}}       
\newcommand{\mc}[1]{{\mathcal #1}}
\newcommand{\mtx}[1]{\left[\,\begin{matrix} #1\end{matrix}\,\right ]}       

\newcommand{\beq}{\begin{equation}}
\newcommand{\eeq}{\end{equation}}
\newcommand{\BEQ}{\begin{equation*}}
\newcommand{\EEQ}{\end{equation*}}

\newcommand{\bea}{\begin{eqnarray}}
\newcommand{\eea}{\end{eqnarray}}
\newcommand{\BEA}{\begin{eqnarray*}}
\newcommand{\EEA}{\end{eqnarray*}}

\newcommand{\bse}{\begin{subequations}}
\newcommand{\ese}{\end{subequations}}
\newcommand{\BSE}{\begin{subequations*}}
\newcommand{\ESE}{\end{subequations*}}

\newcommand{\bca}{\begin{cases}}
\newcommand{\eca}{\end{cases}}

\newcommand{\ben}{\begin{enumerate}}
\newcommand{\een}{\end{enumerate}}

\newcommand{\bit}{\begin{itemize}}
\newcommand{\eit}{\end{itemize}}

\newcommand{\bml}{\begin{multline}}
\newcommand{\eml}{\end{multline}}
\newcommand{\BML}{\begin{multline*}}
\newcommand{\EML}{\end{multline*}}

\newcommand{\rank}{\text{rank}}
\newcommand{\spn}{\text{span}}

\newcommand{\F}{{\mathbb F}}
\renewcommand{\a}{{\mathbf a}}
\renewcommand{\b}{{\mathbf b}}

\newcommand{\e}{{\mathbf e}}

\renewcommand{\u}{{\mathbf u}}
\renewcommand{\v}{{\mathbf v}}
\newcommand{\w}{{\mathbf w}}
\newcommand{\x}{{\mathbf x}}
\newcommand{\y}{{\mathbf y}}

\newcommand{\0}{{\mathbf 0}}

\newcommand{\ch}[1]{\text{char}(#1)}
\newcommand{\Dom}{\text{Dom}}
\newcommand{\Ima}{\text{Im}}

\begin{document}

\title[Extensions of Witt's Theorem]{Some Extensions of Witt's Theorem}
\author{Huajun Huang}
\email{huanghu@auburn.edu}
\address{Department of Mathematics and Statistics\\Auburn University\\AL 36849-5310\\USA}
\classification{11E39 (11E04, 15A63)}
\keywords{Witt's theorem, isometry, flag, Witt's decomposition}

\begin{abstract}
We extend Witt's theorem to several kinds of simultaneous isometries of
 subspaces.
We determine sufficient and necessary conditions  for the extension of an isometry of subspaces $\phi:E\to E'$ to an isometry $\phi_V:V\to V'$ that also sends a given subspace to another,
or a given self-dual flag  to another, or a Witt's decomposition to another
and a special self-dual flag to another.
We also determine   sufficient and necessary conditions   for
the isometry of  generic flags or the simultaneous isometry
of   (subspace, self-dual flag) pairs.
\end{abstract}

\maketitle

\section{Introduction}

Let $\F$ be a field with $\ch\F\ne 2$.
Let $(V,b)$ be a finite dimensional metric space
over $\F$ with a nonsingular symmetric bilinear form
$b:V\times V\to \F$.
The classical Witt's theorem~\cite{E.Witt} states that
every isometry between two subspaces of $V$ can be extended to
an isometry of the whole space $V$.
Witt's theorem has been widely extended
in different context,
such as on the fields of characteristic 2
\cite{MR0008069,MR0015033,MR0241364,MR0168581, MR0182627},
on   various types of local rings
\cite{
MR0249455,
MR0318059,
MR0218381,
MR0280440,
MR0220702,
MR0268750,
MR0310751,
MR539432,
MR548876,
MR0202701,
MR0138565},
on the spaces of
countable dimension
\cite{
MR549828,
MR0195870,
MR0206023,
MR0037285},
on the noncommutative rings
\cite{
MR928496,
MR2033642,
MR0387276},
and on some other situations
\cite{
MR0234913,
MR0242872,
MR0005506}.

The present paper extends the classical Witt's theorem in another direction.
Let $(V, b)$ and $(V', b')$ be isometric nonsingular symmetric metric spaces.
Let $E\subset V$ and $E'\subset V'$
(resp. $A\subset V$ and $A'\subset V'$)
denote isometric subspaces.
Let ${\mc V}$ (resp. ${\mc V}'$) denote a flag of $V$ (resp. $V'$),
that is, an ascending chain of subspaces of $V$.
We provide:
\ben
\item
A sufficient and necessary condition for
an  isometry $\phi:E\to E'$ to be extended to an isometry
of the whole space $\phi_{V}:V\to V'$
that satisfies one of the following:
\ben
\item
$\phi_V$ sends a given subspace $A\subset V$ to another given subspace
$A'\subset V'$ (Theorem \ref{Witt-main}).

\item
$\phi_V$ sends a given self-dual flag ${\mc V}$  to
another given self-dual flag ${\mc V}'$
(Theorem \ref{thm:iso-Witt-ext}).

\item
$\phi_V$ sends a given Witt's decomposition $V=V^+\oplus\widehat V\oplus V^-$ of $V$
to another given Witt's decomposition
$V'=V'^+\oplus \widehat{V'}\oplus V'^-$ of $V'$, and sends
a given self-dual flag ${\mc V}$ ``compatible'' with $V^+$
to another given self-dual flag ${\mc V}'$ ``compatible'' with $V'^+$
(Theorem \ref{thm:Witt-decomposition}).

\een

\item
A sufficient and necessary condition for the existence of an isometry
$\phi_V:V\to V'$ that satisfies one of the following:
\ben
\item
$\phi_V$ sends a given generic flag ${\mc V}$   to
another given generic  flag ${\mc V}'$   (Theorem \ref{Witt-flag}).
\item
$\phi_V$ sends a given subspace $E\subset V$ to a given subspace $E'\subset V'$,
and a given self-dual flag ${\mc V}$ of $V$ to a given self-dual flag
${\mc V}'$ of $V'$ (Theorem \ref{thm:iso-Witt-like}).
\een

\een

The work has the merit of embracing all classical forms---
analogous results hold when $b$ and $b'$  are
alternating, Hermitian, or skew-Hermitian forms, and the proofs are similar.
It is interesting to extend the results to other
algebraic structures like skew-fields, local rings, or to characteristic 2 situations.

\section{Preliminary}\label{Preliminary}

Let $(V, b)$ and $(V', b')$ be isometric
finite dimensional symmetric metric spaces over $\F$,
{\em either singular or nonsingular} (in this section only). This means that
there is a linear bijection $\phi_{V}:V\to V'$ such that
$b(\u,\v)=b'(\phi_V(\u),\phi_V(\v))$ for all $\u,\v\in V$.
We call  $\phi_{V}$ an {\em isometry}.

We adopt some conventional notations from \cite{MR2125693}.
The metric space $(V,b)$ is {\em anisotropic}
if $b(\v,\v)\ne 0$ for all nonzero vector $\v\in V$.
It is {\em totally isotropic} if $b(\v,\v)=0$ for all
$\v\in V$.
Let $A^\perp$ denote the orthogonal complement of $A$.
The {\em radical} of a metric space (or meric subspace) $A$ is $A^\perp\cap A$, or simply $A^\perp$ if $A$ is the whole space.
Obviously, a radical is  totally isotropic.
Let $A\oplus B$ denote the direct sum, and
$A\odot B$ denote the {\em orthogonal} direct sum,
of subspaces $A$ and $B$ of a metric space.
Let $A-B:=\{\v\in A\mid \v\notin B\}$ (which is unlike
the notation of $A+B$).
Write $A\approx A'$ when $A$ and $A'$ are isometric
with respect to the corresponding metrics,
and $A\overset{\phi}{\approx}A'$ when $A$ and $A'$ are
isometric via an isometry $\phi$
(the domain of $\phi$ may be larger than $A$).

The following two  lemmas are handy in the later computations.

\begin{lem}\label{space} Let $A, B, C, E$ be subspaces of
the metric space $(V, b)$.
\ben
\item[(1)]
$\dim (A+B)+\dim (A\cap B)=\dim A+\dim B$.
\item[(2)]
If $V$ is nonsingular, then $\dim A+\dim A^\perp=\dim V$.
\item[(3)]
$(A+B)^\perp=A^\perp\cap B^\perp$ and $(A\cap B)^\perp= A^\perp+B^\perp$.
\item[(4)]
 $A\subset B$ if and only if $A^\perp\supset B^\perp$.
\item[(5)]
If $B\subset C$, then $(A+B)\cap C=A\cap C+B.$
\item[(6)]
If $A\subset E$, then $A\cap B= A\cap E \cap B.$
\item[(7)]
If $A\subset E$ and $B\subset E$, then $A\cap (B+C)=A\cap (B+C\cap E).$
\item[(8)]
$\dim A-\dim (A\cap B^\perp)=\dim B-\dim (B\cap A^\perp)$.
\een
\end{lem}

\begin{lem} \label{isometry}
Suppose $A, B\subset V$ and $A', B'\subset V'$ are metric subspaces.
\ben
\item[(1)]
If $A\overset{\phi}{\approx} A'$, then
every isometry $\phi_{V}:V\to V'$ extended from $\phi$
satisfies
$A^\perp\overset{\phi_{V}}{\approx} A'^\perp$.
\item[(2)]
If $A\overset{\phi}{\approx} A'$ and $B\overset{\phi}{\approx} B'$,
then $A\cap B\overset{\phi}{\approx} A'\cap B'$ and
$A+B\overset{\phi}{\approx} A'+B'$.
\een
\end{lem}

All proofs  are trivial except for that of Lemma \ref{space}(8).
Select a basis of $A\cap B^\perp$ and extend it to a basis
$\{\x_1,\cdots,\x_{\dim A}\}$ of $A$.
Select a basis of $B\cap A^\perp$ and extend it to a basis
$\{\y_1,\cdots,\y_{\dim B}\}$ of $B$. Then
Lemma \ref{space}(8) comes from the   equalities:
$$
\dim A-\dim (A\cap B^\perp)
=\rank \mtx{b(\x_i, \y_j)}_{\dim A\times\dim B}
=\dim B-\dim (B\cap A^\perp).
$$

Given linear maps $\mu:A\to A'$ and $\psi: B\to B'$
where $A, B\subset V$ and $A', B'\subset V'$, we define
$\mu\oplus\psi$  naturally
whenever $A\oplus B$ and $A'\oplus B'$ are well-defined.
Likewise for $\mu\odot \psi$.
More generally, if $\mu$ and $\psi$ agree on the intersection of domains
$\Dom(\mu)\cap\Dom(\psi)=A\cap B,$
then there is a unique linear map $\rho$
from $A+B$ to $A'+B'$ that naturally extends both $\mu$ and $\psi$.
Call $\rho$ the {\em combination} of $\mu$ and $\psi$.
A quick observation shows that if both $\mu$ and $\psi$ are linear bijections,
then so is the combination of $\mu$ and $\psi$.
Nevertheless, if both $\mu$ and $\psi$ are isometries, then
the combination of $\mu$ and $\psi$  may not necessarily be an isometry.

Witt's theorem can be slightly extended to include singular cases.

\begin{lem}\label{thm:Witt-deg}
Every isometry $\phi:E\to E'$
from subspace $E$ of $V$ to subspace $E'$ of $V'$ with
$\phi(E\cap V^\perp)=E'\cap V'^\perp$ can be extended to an isometry
from $V$ to $V'$.
\end{lem}

\begin{proof}
Select subspace $\widetilde E$ such that $E=(E\cap V^\perp)\odot\widetilde E$.
Extend $\widetilde E$ to $\widetilde V$ such that
$V=V^\perp\odot\widetilde V$, where $V^\perp$ is the radical of $V$.
Denote $\widetilde E':=\phi(\widetilde E)$.
Then
$$
E'=\phi(E\cap V^\perp)\odot\phi(\widetilde E)
=(E'\cap V'^\perp)\odot\widetilde E'.
$$
Extend $\widetilde E'$
to  $\widetilde V'$ such that $V'=V'^\perp\odot\widetilde V'$.
From $V\approx  V'$ we get
$\widetilde V\approx V/V^\perp\approx V'/V'^\perp\approx \widetilde V'$,
where $\widetilde V$ and $\widetilde V'$ are nonsingular.
By Witt's theorem (with respect to $V$ and $V'$) the isometry
$\phi|_{\widetilde E}$ can be extended to
an isometry $\phi_{\widetilde V}:\widetilde V\to \widetilde V'$.
Moreover,
since $\phi(E\cap V^\perp)=E'\cap V'^\perp$
and $\dim V^\perp=\dim V'^\perp$,
the isometry $\phi|_{E\cap V^\perp}$ can be extended to
  an isometry of the totally isotropic subspaces $\phi_{V^\perp}: V^\perp\to V'^\perp$.
Thus $\phi_{V}:=\phi_{V^\perp}\odot\phi_{\widetilde V}$
is an isometry from $V$ to $V'$ that extends $\phi$.
\end{proof}

There is a parallel result with no initial isometry provided.

\begin{lem} \label{E-V-*}
Let $E\subset V$ and $E'\subset V'$
such that $E\approx E'$
and $\dim(E\cap V^\perp)=\dim(E'\cap V'^\perp)$.
Then there exists an isometry
$\phi_{V}:V\to V'$ with $\phi_{V}(E)=E'$.
\end{lem}

\begin{proof}
Construct a linear bijection (i.e. an isometry)
$\phi_0: E\cap V^\perp\to E'\cap V'^\perp$.
Then
$$\phi_0((E\cap V^\perp)\cap E^\perp)=\phi_0(E\cap V^\perp)
=E'\cap V'^\perp=(E'\cap V'^\perp)\cap E'^\perp.$$
Apply  Lemma \ref{thm:Witt-deg} on isometric metric spaces $E\approx E'$
and the isometry $\phi_0$. Then $\phi_0$ can be extended to
an isometry $\phi_1: E\to E'$ with $\phi_1(E\cap V^\perp)=E'\cap V'^\perp$. Apply Lemma \ref{thm:Witt-deg} again
on $V\approx V'$ and the isometry $\phi_1$.
Then $\phi_1$ can be extended to an isometry
$\phi_{V}: V\to V'$. Clearly $\phi_{V}(E)=E'$.
\end{proof}

Lemma \ref{thm:Witt-deg}
and Lemma \ref{E-V-*} lead to the  following results.

\begin{cor}\label{E-A-perp}
Suppose  $V\approx V'$ are nonsingular.
Let $E, A\subset V$ and $E', A'\subset V'$ satisfy that
$E\perp A$, $E'\perp A'$, $A\approx A'$, and
$\phi:E\to E'$ is an isometry with
$\phi(E\cap A)=E'\cap A'$. Then $\phi$ can be extended to an isometry
$\phi_{V}:V\to V'$ that sends $A$ to $A'$.
\end{cor}

\begin{proof}
  Notice that $(A^\perp)^\perp=A$ and $(A'^\perp)^\perp=A'$.
Apply Lemma \ref{thm:Witt-deg} on subspace $E$ of $A^\perp$ and
subspace $E'$ of $A'^\perp$.
Then the isometry $\phi$ can be extended to an isometry
$\phi_{1}:A^\perp\to A'^\perp$, which can be further
extended to an isometry $\phi_{V}:V\to V'$ by Witt's theorem.
Then $\phi_{V}$ sends $A$ to $A'$.
\end{proof}

\begin{cor}\label{E-A-perp-*}
Suppose  $V\approx V'$ are nonsingular.
Let $E, A\subset V$ and $E', A'\subset V'$ satisfy that
$E\perp A$, $E'\perp A'$, $E\approx E'$, $A\approx A'$, and
$\dim(E\cap A)=\dim(E'\cap A')$. Then there exists an isometry
$\phi_{V}:V\to V'$ such that
$E\overset{\phi_{V}}{\approx}E'$ and $A\overset{\phi_{V}}{\approx}A'$.
\end{cor}

\begin{proof}
Again $(A^\perp)^\perp=A$ and $(A'^\perp)^\perp=A'$.
Applying Lemma \ref{E-V-*} on subspace $E$ of $A^\perp$ and
subspace $E'$ of $A'^\perp$,
we get an isometry $\phi_{1}:A^\perp\to A'^\perp$ that sends $E$ to $E'$.
It can be extended to
an isometry $\phi_{V}:V\to V'$ by Witt's theorem.
Then $E\overset{\phi_{V}}{\approx}E'$ and $A\overset{\phi_{V}}{\approx}A'$.
\end{proof}

\section{Simultaneous isometry of
(subspace, self-dual flag) pairs}\label{sect:self-dual-I}

From now on, let $(V, b)\approx (V', b')$ be {\em nonsingular}
isometric symmetric metric spaces over $\F$.
In this section, we determine when
an isometry $\phi:E\to E'$ between subspaces $E\subset V$ and $E'\subset V'$
can be extended to an isometry
 $\phi_{V}:V\to V'$  that sends a
given self-dual flag of $V$ to another given self-dual flag
of $V'$.

Let $\0$ (resp. $\0'$) denote the zero vector in $V$ (resp. $V'$).

A {\em flag} ${\mc V}$ of   $V$ is
a chain of embedding subspaces of $V$:
\beq
{\mc V}:=\left\{V_0\subset V_1\subset \cdots\subset V_k \right\}.
\eeq
For convenience, we set $V_0$ the zero space and $V_k$ the whole space throughout the paper,
and write ${\mc V}=\{V_i\}_{i=0,\cdots,k}$.
The flag ${\mc V}$  is called a {\em self-dual flag} if
\beq
\{V_k^\perp\subset V_{k-1}^\perp\subset\cdots\subset V_0^\perp\}
=\{V_0\subset V_1\subset \cdots\subset V_k\}.
\eeq
When ${\mc V}$ is a self-dual flag,
the subspaces in the ``first half part'' of ${\mc V}$ are
 totally isotropic.

Let ${\mc V}=\{V_i\}_{i=0,\cdots,k}$ (resp. ${\mc V}'=\{V_i'\}_{i=0,\cdots,k'}$)
be a flag of $V$ (resp. $V'$).
Given a linear map $\phi_V:V\to V'$, we write $\phi_V({\mc V})={\mc V}'$ if
$k=k'$ and $\phi_V(V_i)=V_i'$ for $i=0,\cdots,k$.
Denote ${\mc V}\approx {\mc V}'$ or ${\mc V}\overset{\phi_V}{\approx}{\mc V}'$
(when $\phi_V$ is given)
if there is an isometry $\phi_V:V\to V'$ such that $\phi_V({\mc V})={\mc V}'$.

\begin{lem}
Suppose ${\mc V}$ and ${\mc V}'$ defined above are self-dual flags.
Then ${\mc V}\approx {\mc V}'$ if and only if $k=k'$ and
$\dim V_i=\dim V_i'$ for $i=1,\cdots,\lfloor\frac k2\rfloor$.
\end{lem}

\begin{proof}
Only the sufficient part is nontrivial.
By the sufficient conditions in the lemma, we can construct a linear bijection
$\phi: V_{\lfloor\frac k2\rfloor}\to V'_{\lfloor\frac k2\rfloor}$
such that $\phi(V_i)=V_i'$ for $i=1,\cdots,\lfloor\frac k2\rfloor$.
Then $\phi$ is an isometry of totally isotropic subspaces.
By Witt's theorem $\phi$ can be extended to an isometry
$\phi_{V}:V\to V'$.
Clearly ${\mc V}\overset{\phi_V}{\approx}{\mc V}'$.
\end{proof}

The criteria for the isometry  of two generic flags
will be determined in Section \ref{sect:flag}.

Let ${\mc V}$ and ${\mc V}'$ be generic flags with
$k=k'$.
Given $E\subset V$ and $E'\subset V'$, we denote
\beq\label{E_iE_i'}
E_i:=E\cap V_i,\qquad E_i':=E'\cap V_i',\qquad
\text{for}\quad i=0,\cdots, k.
\eeq

\begin{thm}\label{thm:iso-Witt-ext}
Let ${\mc V}:=\{V_i\}_{i=0,\cdots,k}$ and
${\mc V}':=\{V_i'\}_{i=0,\cdots,k}$ be isometric self-dual flags of
$V$ and $V'$ respectively.
Then an isometry $\phi:E\to E'$ from $E\subset V$ to $E'\subset V'$ can be extended to an isometry
$\phi_{V}: V\to V'$ that sends ${\mc V}$ to ${\mc V}'$,
if and only if $\phi(E_i)=E_i'$ for $i=0,\cdots,k.$
\end{thm}

The result plays a key role in the author's classification
of Borel subgroup orbits of classical symmetric subgroups
on multiplicity-free flag varieties (\cite{huang01, huang02, huang03}, preprints).
The cases of  $k\le 2$ may be viewed as special cases  of $k=3$
by adding certain repeated subspace(s) on
  self-dual flags of $k\le 2$.
Surprisedly, $k=3$ implies the general cases.

\begin{lem}\label{thm:iso-Witt-gen}
For $k=3$, Theorem \ref{thm:iso-Witt-ext} is true.
\end{lem}

\begin{proof}[Proof of Theorem \ref{thm:iso-Witt-ext}
by Lemma \ref{thm:iso-Witt-gen}]
It suffices to prove the sufficient part.
\ben
\item
Consider the isometry $\phi: E\to E'$ and the self-dual flags
$$
\left\{\{\0\}\subset V_1\subset V_{k-1}\subset V\right\}
\qquad\text{and}\qquad
\left\{\{\0'\}\subset V_1'\subset V_{k-1}'\subset V'\right\}.
$$
Since $\phi(E_i)=E_i'$ for $i=0,1,k-1,k$,
by Lemma \ref{thm:iso-Witt-gen} the isometry $\phi$ can be extended to
an isometry $\phi_1:V\to V'$ that sends $V_1$ to $V_1'$.
The restricted map
$\phi_1|_{E+V_1}$ is an isometry from
$E+V_1$ to $E'+V_1'$,
which sends $E$ to $E'$ and $V_1$ to $V_1'$ respectively.

\item
Consider the isometry
 $\phi_1|_{E+V_1}$ and the self-dual flags
$$
\left\{\{\0\}\subset V_2\subset V_{k-2}\subset V\right\}
\qquad\text{and}\qquad
\left\{\{\0'\}\subset V_2'\subset V_{k-2}'\subset V'\right\}.
$$
Clearly
$V_1\subset V_i$ and $V_1'\subset V_i'$ for $i=2,k-2$.
By Lemma \ref{space}(5),
\BEA
\phi_1|_{E+V_1}\left((E+V_1)\cap V_i\right)
&=& \phi_1(E\cap V_i+V_1)
=\phi(E_i)+\phi_1(V_1)
\\
&=& E_i'+V_1'=(E'+V_1')\cap V_i'
\qquad\text{for}\quad i=2,k-2.
\EEA
By Lemma \ref{thm:iso-Witt-gen}, the map $\phi_1|_{E+V_1}$ can be extended to
an isometry $\phi_2:V\to V'$ that sends $V_2$ to $V_2'$.
The restricted map $\phi_2|_{E+V_2}$ is an isometry from $E+V_2$ to $E'+V_2$, which
sends $E$ to $E'$, $V_1$ to $V_1'$, and $V_2$ to $V_2'$ respectively.

\item
Repeat the above process.
Eventually, $\phi:E\to E'$ can be extended to an isometry
$\phi_{\lfloor\frac k2\rfloor}: V\to V'$
that sends $V_i$ to $V_i'$ for $i=1,\cdots, \lfloor\frac k2\rfloor$.
Then $\phi_{V}:=\phi_{\lfloor\frac k2\rfloor}$ is an extension of $\phi$,
with $\phi_{V}(V_i)=V_i'$ and
$\phi_{V}(V_{k-i})=\phi_{V}(V_i^\perp)=\phi_{V}(V_i)^\perp=V'_{k-i}$
for $i=1,\cdots, \lfloor\frac k2\rfloor$.
This completes the proof.
\een
\end{proof}

\begin{proof}[Proof of Lemma \ref{thm:iso-Witt-gen}]
Only the sufficient part is nontrivial.
Suppose
$$
{\mc V}:=\{V_i\}_{i=0,\cdots,3}
\qquad\text{and}\qquad
{\mc V}':=\{V_i'\}_{i=0,\cdots,3}
$$
are isometric self-dual flags of $V$ and $V'$ respectively,
and the isometry $\phi:E\to E'$ satisfies that
$\phi(E_i)=E_i'$ for $i=0,1,2,3$.
We extend  $\phi$ to an isometry
$\phi_1:E+V_1\to E'+V_1'$
by the self-explained diagram:
\[
\xymatrix{
E_1 \ar@{^{(}->}[r] \ar@{->}[d]_{\phi|_{V_1}}
    &E_1+(V_1\cap E^\perp) \ar@{^{(}->}[r] \ar@{->}[d]_{\phi_0|_{V_1}}
            &V_1 \ar@{->}[d]_{\phi_1|_{V_1}}
\\
E_1' \ar@{^{(}->}[r]
    &E_1'+(V_1'\cap E'^\perp) \ar@{^{(}->}[r]
            &V_1'
}
\]
The total isotropies of $V_1$ and $V_1'$ make the constructions of
isometries much easier.

\ben
\item Extend $\phi: E\to E'$ to
an isometry $\phi_0: E+(V_1\cap E^\perp)\to E'+(V'_1\cap E'^\perp)$.

By Lemma \ref{space},
\BEA
\dim (V_1\cap E^\perp)
&=& \dim V_1+\dim E^\perp-\dim (V_1+E^\perp)
\\
&=& \dim V_1+(\dim V-\dim E)-[\dim V-\dim (V_1+E^\perp)^\perp]
\\
&=& \dim V_1-\dim E+\dim (V_1^\perp\cap E^{\perp\perp})
\\
&=& \dim V_1-\dim E+\dim E_2.
\EEA
Similar computation on $\dim (V_1'\cap E'^\perp)$
shows that $\dim (V_1\cap E^\perp)=\dim (V_1'\cap E'^\perp)$.
Then $V_1\cap E^\perp\approx V_1'\cap E'^\perp$ since both spaces
are totally isotropic. Notice that
$$
E\cap (V_1\cap E^\perp) = E_1\cap E^\perp,
\quad
E'\cap (V_1'\cap E'^\perp) = E_1'\cap E'^\perp,
\quad
\phi(E_1\cap E^\perp) = E_1'\cap E'^\perp.
$$
Applying Corollary \ref{E-A-perp} on $\phi: E\to E'$,
$A:=V_1\cap E^\perp$ and $A':=V_1'\cap E'^\perp$,
we get an isometric extension of $\phi$ that sends
$V_1\cap E^\perp$ to $V_1'\cap E'^\perp$.
That is, $\phi$ can be extended to an isometry
$$
\phi_0:
E+(V_1\cap E^\perp)\to E'+(V'_1\cap E'^\perp).
$$

\item Extend $\phi_0: E+(V_1\cap E^\perp)\to E'+(V'_1\cap E'^\perp)$ to
an isometry $\phi_1: E+V_1\to E'+V'_1$.

In $V$,
select $\widetilde E$ such that $E=E_2\oplus \widetilde E$.
Select $\overline E\subset E_1$ and expand it to $\overline V\subset V_1$
 such that
\beq\label{E_1-V_1}
E_1=(E_1\cap E^\perp)\odot \overline E,\qquad
V_1=(V_1\cap E^\perp)\odot \overline V.
\eeq
Then by Lemma \ref{space}(8),
\beq\label{dim-overline-V}
\dim \overline V
=\dim V_1-\dim (V_1\cap E^\perp)
=\dim E-\dim (E\cap V_1^\perp)
=\dim\widetilde E.
\eeq

The Riesz representation  $\tau_b: \overline V\to (\widetilde E)^*$,
defined by
$$
\tau_b(\v)(\e):=b(\e,\v)
\quad\text{for}\quad
\v\in \overline V,\ \ \e\in \widetilde E,
$$
 is injective since $\overline V\subset V_1= V_2^\perp\subset E_2^\perp$ and so
$$\ker\tau_b=\overline V\cap\widetilde E^\perp=\overline V\cap E_2^\perp\cap \widetilde E^\perp
=\overline V\cap (E_2+\widetilde E)^\perp=\overline V\cap E^\perp=\{\0\}.$$
Hence $\tau_b$ is an isomorphism
by \eqref{dim-overline-V}.

In $V'$,
let $\widetilde E':=\phi_0(\widetilde E)$ so that $E'=E_2'\oplus\widetilde E'$;
Let $\overline E':=\phi_0(\overline E)$.
By \eqref{E_1-V_1},
$$
E'_1 =\phi_0(E_1\cap E^\perp)\odot \phi_0(\overline E)=(E'_1\cap E'^\perp)\odot\overline E'.
$$
Expand $\overline E'$ to $\overline V'\subset V_1'$ such that
$
V_1'=(V_1'\cap E'^\perp)\odot \overline V'.
$
Likewise, $\dim \overline V'=\dim \widetilde E'$.
The Riesz representation $\tau_{b'}:\overline V'\to (\widetilde E')^*$,
defined by
$$\tau_{b'}(\v')(\e'):=b'(\e', \v')
\quad\text{for}\quad
\v'\in\overline V', \ \ \e'\in\widetilde E',
$$
is also an isomorphism.

We now construct an isometry $\phi_{\beta}:\overline V\to\overline V'$
such that $\phi_0$ and $\phi_\beta$ agree on the intersection of domains,
and the combination $\phi_1$ of $\phi_0$ and $\phi_\beta$ is an isometry.

Let $\phi_{\beta}$
be induced by the commutative diagram:
\[
\xymatrixcolsep{4pc}\xymatrix{
\overline V \ar@{->}[r]^{\tau_b}  \ar@{->}[d]_{\phi_{\beta}}
    &(\widetilde E)^* \ar@{->}[d]^{(\phi_0^{-1}|_{\widetilde E'})^*}
\\
\overline V' \ar@{->}[r]^{\tau_{b'}}
    &(\widetilde E')^*
}
\]
that is, $\phi_{\beta}:\overline V\to\overline V'$ is uniquely determined by
\beq\label{phi_1}
b'\left(\phi_0(\e), \phi_{\beta}({\mb v})\right)
=b(\e,  {\mb v})
\qquad
\text{for \ }
{\mb v}\in \overline V
\ \text{ and } \ \e\in \widetilde E.
\eeq
Then $\phi_{\beta}=
{\tau_{b'}}^{-1}\circ (\phi_0^{-1}|_{\widetilde E'})^*\circ\tau_b$
is an isometry between totally isotropic spaces $\overline V$ and $\overline V'$.

By Lemma \ref{space}(5) and \eqref{E_1-V_1},
\[
\Dom(\phi_0)\cap V_1
= \left[E+(V_1\cap E^\perp)\right]\cap V_1
=E_1+(V_1\cap E^\perp)
=\overline E\odot(V_1\cap E^\perp).
\]
So by \eqref{E_1-V_1}, $\overline V\subset V_1$ and thus
$$\Dom(\phi_0)\cap\Dom(\phi_{\beta})
=\Dom(\phi_0)\cap V_1\cap\overline V=\overline E.$$
Given $\overline\e\in\overline E$,
we have $\phi_0(\overline\e)\in\overline E'\subset \Ima(\phi_{\beta})$,
and by \eqref{phi_1}
$$
b'\left(\phi_0(\e), \phi_{\beta}(\overline\e)\right)
=b(\e, \overline\e)
=b'\left(\phi_0(\e), \phi_0(\overline\e)\right)
\quad\text{for all}\quad
\e\in\widetilde E.
$$
So
the isometries $\phi_0$ and $\phi_{\beta}$ agree on
$\overline E$, the intersection of domains.
Therefore, the combination $\phi_1$ of $\phi_0$ and $\phi_{\beta}$ is
 a linear bijection and
$$
\Dom(\phi_1)=
\Dom(\phi_0)+\Dom(\phi_{\beta})
=E+(V_1\cap E^\perp)+\overline V=E+V_1.
$$
Likewise, $\Ima(\phi_1)=E'+V_1'$.

To prove that $\phi_1$ is an isometry, it remains to show that
\beq\label{b'b}
b'(\phi_1(\u),\phi_1(\v))= b(\u,\v)
\eeq
for $\u\in\Dom(\phi_0)$ and $\v\in\Dom(\phi_{\beta})$.
By
$
\widetilde E\cap [E_2+(V_1\cap E^\perp)]
\subset \widetilde E\cap V_2=\{\0\}
$,
\beq
\Dom(\phi_0)=(\widetilde E\oplus E_2)+(V_1\cap E^\perp)
=\widetilde E\oplus [E_2+(V_1\cap E^\perp)].
\eeq
By
$
E_2+(V_1\cap E^\perp)\subset V_2\subset\overline V^\perp=\Dom(\phi_{\beta})^\perp
$
and
$$
\phi_1\left(E_2+(V_1\cap E^\perp)\right)
=E_2'+(V_1'\cap E'^\perp)\subset V_2'
\subset\overline V'^\perp=\Ima(\phi_{\beta})^\perp,
$$
it suffices to prove \eqref{b'b} for $\u\in\widetilde E$ and
$\v\in\Dom(\phi_{\beta})=\overline V$.
This is exactly \eqref{phi_1}.

Therefore, $\phi_1: E+V_1\to E'+V_1'$ is an isometry, where
$$
\phi_1(V_1)=\phi_0(V_1\cap E^\perp)\odot\phi_{\beta}(\overline V)
=(V_1'\cap E'^\perp)\odot \overline V'=V_1'.
$$
\een

Finally, by Witt's theorem, $\phi_1$ can be extended to
an isometry $\phi_{V}:V\to V'$. Then
$\phi_{V}|_E=\phi$, $\phi_{V}(V_1)=\phi_1(V_1)=V_1'$,
and $\phi_{V}(V_2)=\phi_{V}(V_1^\perp)=\phi_{V}(V_1)^\perp=V_2'$.
\end{proof}

\section{Applications of Theorem \ref{thm:iso-Witt-ext} to Witt's Decomposition}

In this section, we apply Theorem \ref{thm:iso-Witt-ext} to solve the following problem:
When can an isometry of subspaces  be extended to an isometry of
whole spaces that preserves the corresponding Witt's decompositions
and the corresponding self-dual flags ``compatible'' with the
Witt's decompositions?

\begin{lemma}\label{V=A+B}
 Let $V=A\oplus B$ and $V'=A'\oplus B'$ be  vector space decompositions
such that $\dim A=\dim A'$ and $\dim B=\dim B'$. Then a linear bijection
$\phi:E\to E'$ between $E\subset V$ and $E'\subset V'$
can be extended to a  linear bijection
$\phi_V:V\to V'$ such that $\phi_V(A)=A'$ and $\phi_V(B)= B'$  if and only if
$\phi(E\cap A)= E'\cap A'$ and $\phi(E\cap B)= E'\cap B'$.
\end{lemma}

The lemma is a special case of Lemma \ref{thm:A+B+C}. We refer to the proof there. In the above lemma,
every  $\v\in V$ can be uniquely decomposed as
$\v=\v_A+\v_B$ for $\v_A\in A$ and $\v_B\in B$.
Precisely, $\{\v_A\}=(\{\v\}+B)\cap A$ and $\{\v_B\}=(\{\v\}+A)\cap B$.
Likewise,
every   $\v'\in V'$ can be uniquely decomposed as
$\v'=\v_{A'}'+\v_{B'}'$ for $\v_{A'}'\in A'$ and $\v_{B'}'\in B'$.
Denote
$P_{A,B}(E):=\left[(E+B)\cap A\right]\oplus\left[(E+A)\cap B\right].$
Then $P_{A,B}(E)\supset E$.  Define a linear map
$\widetilde\phi_{A,B}:P_{A,B}(E)\to P_{A',B'}(E')$ by
\beq\label{widetilde-phi-def}
\widetilde\phi_{A,B}(\v_A):=\phi(\v)_{A'},\qquad
\widetilde\phi_{A,B}(\v_B):=\phi(\v)_{B'},\qquad
\text{for all}\quad \v\in E.
\eeq
From the proof of Lemma \ref{thm:A+B+C},
$\widetilde\phi_{A,B}$ is a well-defined bijective linear extension of
$\phi$. Moreover, every linear extension $\phi_V:V\to V'$ of $\phi$
that satisfies $\phi_V(A)=A'$ and $\phi_V(B)=B'$ is also an extension
of $\widetilde\phi_{A,B}$.

A {\em hyperbolic space}
is a nonsingular  metric space $V$ that has the
Witt decomposition $V=V^+\oplus V^-$, where both
$V^+$ and $V^-$ are maximal totally isotropic subspaces
of the same dimension  \cite{MR2125693}.
If ${\mc V}:=\{V_i\}_{i=0,\cdots,k}$ is a self-dual flag of $V$
such that $V_{\lfloor\frac k2\rfloor}\subset V^+$,
then $V_i\subset V^+$ for $i=0,\cdots,\lfloor\frac k2\rfloor$
and $V^+\subset V_i$ for $i=\lfloor\frac k2\rfloor+1,\cdots, k$.
There are two cases about ${\mc V}$:
\ben
\item $k$ even: Then $V_{\lfloor\frac k2\rfloor}=V_{\frac k2}=V^+$. Denote
 $\overline{\mc V}:={\mc V}$.
\item $k$ odd: Then
$\overline{\mc V}:=\{V_0, V_1,\cdots, V_{\lfloor\frac k2\rfloor}, V^+, V_{\lfloor\frac k2\rfloor+1},
\cdots, V_k\}$ is also a self-dual flag. It is a refinement of ${\mc V}$.

\een

The following theorem is a slight improvement of a
main theorem in \cite{huang02}.

\begin{thm}\label{thm:hyperspace}
Let $V=V^+\oplus V^-$ and $V'=V'^+\oplus V'^-$ be isometric hyperbolic spaces. Let ${\mc V}:=\{V_i\}_{i=0,\cdots,k}$ and
${\mc V}':=\{V_i'\}_{i=0,\cdots,k}$ be isometric self-dual flags of $V$ and
$V'$ respectively
such that $V_{\lfloor\frac k2\rfloor}\subset V^+$
and $V_{\lfloor\frac k2\rfloor}'\subset V'^+$.
Then an isometry $\phi:E\to E'$ between $E\subset V$ and $E'\subset V'$ can be extended
to an isometry $\phi_{V}:V\to V'$ that satisfies $\phi_V(V^+)=V'^+$,
$\phi_V(V^-)=V'^-$, and $\phi_V({\mc V})={\mc V}'$,
 if and only if
$\phi(E\cap V^+)=E'\cap V'^+$,
$\phi(E\cap V^-)=E'\cap V'^-$,
and the induced linear bijection $\widetilde\phi_{V^+, V^-}$ in \eqref{widetilde-phi-def}
is an isometry that satisfies
$$
\widetilde\phi_{V^+,V^-}(P_{V^+, V^-}(E)\cap V_i)=P_{V'^+, V'^-}(E')\cap V_i'
\qquad\text{for}\qquad i=1,\cdots,k.
$$
\end{thm}

\begin{proof}
It suffices to prove the sufficient part.
Since $\phi$ can be extended to $\widetilde \phi_{V^+,V^-}$ by Lemma \ref{V=A+B}
and the isometry $\widetilde\phi_{V^+,V^-}:P_{V^+,V^-}(E)\to P_{V'^+, V'^-}(E')$ satisfies the sufficient conditions
about $\phi:E\to E'$ in the above theorem,
without lost of generality, we assume that $\phi=\widetilde\phi_{V^+,V^-}$ so that \BEA
E &=& P_{V^+, V^-}(E)=(E\cap V^+)\oplus (E\cap V^-),
\\
E' &=& P_{V'^+, V'^-}(E')=(E'\cap V'^+)\oplus(E'\cap V'^-).
\EEA
Moreover, the sufficient conditions in the theorem still hold if we replace ${\mc V}$ by
$\overline{\mc V}$ and ${\mc V}'$ by $\overline{{\mc V}'}$.
Thus we assume  ${\mc V}=\overline{\mc V}$
and ${\mc V}'=\overline{{\mc V}'}$  so that $k$ is even.

By
Theorem \ref{thm:iso-Witt-ext},
the isometry $\phi:E\to E'$ can be extended to an isometry
$\phi_V':V\to V'$ such that $\phi_V'({\mc V})={\mc V}'$.
Then $\phi_V'(V^+)=\phi_V'(V_{\frac k2})=V_{\frac k2}'=V'^+$.
Then $\phi_1:=\phi_V'|_{E+V^+}$ is an isometric extension of $\phi$ sending $E+V^+$ to $E'+V'^+$.
Moreover, $\phi_1(V_i)=V_i'$ for $i=1,\cdots,\frac k2$ (in particular $\phi_1(V^+)=V'^+$).
Notice that
$$
\phi_1((E+V^+)\cap V^-)=\phi(E\cap V^-)=E'\cap V'^-=(E'+V'^+)\cap V'^-.
$$
Applying Theorem \ref{thm:iso-Witt-ext}  on $\phi_1:E+V^+\to E'+V'^+$
and the self-dual flags $\{\{\0\}\subset V^-\subset V\}$
and $\{\{\0'\}\subset V'^-\subset V'\}$, the isometry
$\phi_1$ can be extended to an isometry $\phi_V:V\to V'$ such that
$\phi_V(V^-)=V'^-$. Finally, $\phi_V(V_i)=\phi_1(V_i)=V_i'$ for $i=1,\cdots,\frac k2$ and
$$\phi_V(V_i)=\phi_V(V_{k-i}^\perp)=\phi_V(V_{k-i})^\perp
=V_{k-i}'^\perp=V_i'$$
for $i=\frac k2+1,\cdots, k$. This shows that $\phi_V({\mc V})={\mc V}'$.
We are done.
\end{proof}

\begin{lemma}\label{thm:A+B+C}
Suppose $V=A\oplus B\oplus C$ and $V'=A'\oplus B'\oplus C'$
  satisfy that $\dim A=\dim A'$, $\dim B=\dim B'$, and $\dim C=\dim C'$.
  Then a linear bijection $\phi:E\to E'$ can be extended to a linear bijection
  $\phi_V:V\to V'$ such that $\phi_V(A)=A'$, $\phi_V(B)=B'$, and
  $\phi_V(C)=C'$, if and only if
\bse\label{A+B+C}
\bea
\phi(E\cap (B+C)) &=& E'\cap (B'+C'),
\label{B+C}\\
\phi(E\cap (C+A)) &=& E'\cap (C'+A'),
\label{C+A}\\
\phi(E\cap (A+B)) &=& E'\cap (A'+B').
\label{A+B}
\eea
\ese
\end{lemma}

\begin{proof}
  It suffices to prove the sufficient part.
Every $\v\in V$ can be uniquely expressed as $\v=\v_A+\v_B+\v_C$
for $\v_A\in A$, $\v_B\in B$, $\v_C\in C$. Precisely,
$\{\v_A\}=(\{\v\}+B+C)\cap A$ and likewise for $\v_B$ and $\v_C$.
Similarly, every $\v'\in V'$ can be uniquely expressed as
$\v'=\v_{A'}'+\v_{B'}'+\v_{C'}'$
for $\v_{A'}'\in A'$, $\v_{B'}'\in B'$, $\v_{C'}'\in C'$.
From \eqref{B+C}, we claim that the map $\phi_{A}'(\v_A):= \phi(\v)_{A'}$ for
$\v\in E$ is a well-defined linear bijection.
To see the well-definedness, suppose $\u, \v\in E$ such that
$\u_A=\v_A$, then $\u-\v\in E\cap (B+C)$ and thus $\phi(\u)-\phi(\v)\in E'\cap (B'+C')$ and thus $\phi(\u)_{A'}=\phi(\v)_{A'}$.
Analogous analysis shows that  $\phi_{A}'$ is a bijection.
Then the linearity of $\phi_{A}'$
comes from the linearity of $\phi$.
Likewise, both $\phi_B'(\v_B):=\phi(\v)_{B'}$ for $\v\in E$
and $\phi_C'(\v_C):=\phi(\v)_{C'}$ for $\v\in E$ are  well-defined linear bijections. Extend $\phi_A'$ (resp. $\phi_B'$, $\phi_C'$) to a linear bijection $\phi_A:A\to A'$ (resp. $\phi_B:B\to B'$, $\phi_C:C\to C'$).
Then $\phi_V:=\phi_A\oplus\phi_B\oplus\phi_C$ is a linear bijection from $V$ to $V'$, such that
$\phi_V|_E=\phi$, $\phi_V(A)=A'$, $\phi_V(B)=B'$, and
  $\phi_V(C)=C'$.
\end{proof}

In the above proof, denote the sum of the projections of $E$ onto $A$, $B$, and $C$ components with respect to $V=A\oplus B\oplus C$ by
\beq
P_{A,B,C}(E):= [(E+B+C)\cap A]\oplus [(E+C+A)\cap B]\oplus [(E+A+B)\cap C].
\eeq
When $\phi$ satisfies the conditions in \eqref{A+B+C}, the map $\widetilde\phi_{A,B,C}:=\phi_A'\oplus\phi_B'\oplus\phi_C'$ is
a linear bijection between $P_{A,B,C}(E)$ and $P_{A',B',C'}(E')$.
Alternatively, $\widetilde\phi_{A,B,C}$ may be  defined by
\beq
\widetilde\phi_{A,B,C}(\v_A):=\phi(\v)_{A'},\quad
\widetilde\phi_{A,B,C}(\v_B):=\phi(\v)_{B'},\quad
\widetilde\phi_{A,B,C}(\v_C):=\phi(\v)_{C'},\quad
\text{for}\quad \v\in E.
\eeq
Every extension of $\phi$ to a linear bijection between $V$ and $V'$
that sends $A$ to $A'$, $B$ to $B'$, and $C$ to $C'$ respectively is also a  extension of $\widetilde\phi_{A,B,C}$.

A Witt's decomposition of a generic nonsingular metric space $V$ is
\beq\label{Witt-decomposition}
V=V^+\oplus \widehat V\oplus V^-
\eeq
where $V^+$ and $V^-$ are maximal totally isotropic spaces
of the same dimension  and $\widehat V$ is an anisotropic space orthogonal
to both $V^+$ and $V^-$. Two nonsingular metric spaces are isometric
if and only if their corresponding components in Witt's decompositions
are isometric.

\begin{thm}\label{thm:Witt-decomposition}
  Let $V=V^+\oplus \widehat V\oplus V^-$ and $V'=V'^+\oplus \widehat {V'}\oplus V'^-$ be Witt's decompositions of isometric  spaces $V$ and $V'$ respectively.
Let ${\mc V}=\{V_i\}_{i=0,\cdots,k}$ and ${\mc V}'=\{V_i'\}_{i=0,\cdots,k}$
be isometric self-dual flags of $V$ and $V'$ respectively
such that $V_{\lfloor\frac k2\rfloor}\subset V^+$ and $V_{\lfloor\frac{k}2\rfloor}'\subset V'^+$.
Then an isometry $\phi:E\to E'$ between $E\subset V$ and $E'\subset V'$
can be extended to an isometry $\phi_V:V\to V'$ such that
\beq\label{Witt-decomposition-condition}
\phi_V(\widehat V)=\widehat{V'},\qquad
\phi_V(V^+)=V'^+,\qquad
\phi_V(V^-)=V'^-,\qquad
\phi_V({\mc V})={\mc V}',
\eeq
if and only if
\bse\label{equiv-Witt-decomp}
\bea
\phi(E\cap (V^++V^-)) &=& E'\cap (V'^++V'^-),
\label{equiv-Witt-decomp-a}\\
\phi(E\cap (V^-+\widehat V)) &=& E'\cap (V'^-+\widehat {V'}),
\label{equiv-Witt-decomp-b}\\
\phi(E\cap (\widehat V+V^+)) &=& E'\cap (\widehat{V'}+V'^+),
\label{equiv-Witt-decomp-c}
\eea
and the induced map $\widetilde\phi_{V^+,\widehat V,V^-}: P_{V^+,\widehat V,V^-}(E) \to P_{V'^+,\widehat {V'},V'^-}(E')$ is an isometry
that satisfies
\beq
\widetilde\phi_{V^+,\widehat{V},V^-}(P_{V^+,\widehat V,V^-}(E)\cap V_i)
=P_{V'^+,\widehat {V'},V'^-}(E')\cap V_i' \qquad
\text{for}\qquad i=1,\cdots, k.
\label{equiv-Witt-decomp-d}
\eeq
\ese
\end{thm}

\begin{cor}
An isometry $\phi:E\to E'$ can be extended to an isometry
$\phi_V:V\to V'$ that sends a given Witt's decomposition
$V=V^+\oplus\widehat V\oplus V^-$ to another given Witt's decomposition
$V'=V'^+\oplus\widehat{V'}\oplus V'^-$ if and only if
$\phi$ satisfies \eqref{equiv-Witt-decomp-a}, \eqref{equiv-Witt-decomp-b},
\eqref{equiv-Witt-decomp-c}, and the induced map
$\widetilde\phi_{V^+,\widehat V,V^-}$ is an isometry.
\end{cor}

\begin{proof}[Proof of Theorem \ref{thm:Witt-decomposition}]
It suffices to prove the sufficient part.
If $\dim\widehat V=0$, then the statement is just Theorem \ref{thm:hyperspace}.
Now suppose $\dim\widehat V>0$ so that $k$ is odd.
By the same reasons as stated in the proof of Theorem \ref{thm:hyperspace},
without lost of generality, we assume that
$\phi=\widetilde\phi_{V^+,\widehat V,V^-}$ so that
\BEA
E &=& P_{V^+,\widehat V,V^-}(E)=(E\cap V^+)\oplus (E\cap\widehat V)\oplus (E\cap V^-),
\\
E' &=& P_{V'^+,\widehat {V'},V'^-}(E')=(E'\cap V'^+)\oplus (E'\cap\widehat {V'})\oplus (E'\cap V'^-).
\EEA
Then for $i=\frac{k+1}2,\cdots, k$, by Lemma \ref{space}(5)
$$
V_i=V_i\cap \left[\widehat V\odot (V^+\oplus V^-)\right]=\widehat V\odot \left[V_i\cap(V^+\oplus V^-)\right].
$$
Likewise for $V_i'$ ($i=\frac{k+1}2,\cdots, k$).
Denote $\widetilde {V_i}:=V_i\cap (V^+\oplus V^-)$ and
$\widetilde {V_i'}:=V_i'\cap (V'^+\oplus V'^-)$ for $i=\frac{k+1}2,\cdots, k$.
Then $\widetilde{\mc V}:=\{V_0\subset \cdots\subset V_{\frac{k-1}2}
\subset \widetilde {V_{\frac{k+1}2}}\subset\cdots\subset
\widetilde {V_k}\}$ is a self-dual flag of $V^+\oplus V^-$,
and
$\widetilde{\mc V'}:=\{V_0'\subset \cdots\subset V_{\frac{k-1}2}'
\subset \widetilde {V_{\frac{k+1}2}'}\subset\cdots\subset
\widetilde {V_k'}\}$ is a self-dual flag of $V'^+\oplus V'^-$.
Notice that
$$
\phi=\phi|_{E\cap\widehat V} \odot \phi|_{(E\cap V^+)\oplus (E\cap V^-)}.
$$
By Witt's theorem, the isometry $\phi|_{E\cap\widehat V}: E\cap \widehat V\to
E'\cap\widehat{V'}$ can be extended to an isometry
$\phi_{\widehat V}:\widehat V\to \widehat {V'}$.
By applying Theorem \ref{thm:hyperspace} onto
$\phi|_{(E\cap V^+)\oplus (E\cap V^-)}$ and
the self-dual flags $\widetilde {\mc V}$ and $\widetilde {\mc V'}$,
the isometry
$\phi|_{(E\cap V^+)\oplus (E\cap V^-)}: (E\cap V^+)\oplus (E\cap V^-)
\to (E'\cap V'^+)\oplus (E'\cap V'^-)$ can be extended to
an isometry $\phi_{V^+\oplus V^-}: V^+\oplus V^-\to V'^+\oplus V'^-$
such that $\phi_{V^+\oplus V^-}(V^+)=V'^+$,
$\phi_{V^+\oplus V^-}(V^-)=V'^-$, and $\phi_{V^+\oplus V^-}(\widetilde {\mc V})=\widetilde {\mc V'}$.
Then $\phi_V:=\phi_{\widehat V}\odot\phi_{V^+\oplus V^-}$ is a desired isometric linear extension of $\phi$ that satisfies \eqref{Witt-decomposition-condition}.
\end{proof}

\section{Simultaneous isometry of subspace pairs}

Suppose $(V, b)\approx (V', b')$ are nonsingular.
Let $E, A\subset V$ and $E', A'\subset V'$ satisfy that
$\phi:E\to E'$ is an isometry and $A\approx A'$.
We determine when  $\phi$ can be extended to an isometry
$\phi_{V}: V\to V'$ that
sends $A$ to $A'$.

We begin with the self-dual flags
\beq\label{A-self-dual}
\left\{\{\0\}\subset A^\perp\cap A  \subset A+A^\perp\subset V\right\}
\quad\text{and}\quad
\left\{\{\0'\}\subset A'^\perp\cap A' \subset A'+A'^\perp\subset V'\right\}.
\eeq
By Lemma \ref{thm:iso-Witt-gen},
the isometry $\phi:E\to E'$ can be extended to an isometry
$\phi_{V}': V\to V'$ such that
$A^\perp\cap A \overset{\phi_{V}'}{\approx}A'^\perp\cap A'$ and
$A+ A^\perp \overset{\phi_{V}'}{\approx} A'+ A'^\perp$
if and only if
\begin{quote}
(C1) $E\cap A^\perp\cap A \overset{\phi}{\approx}E'\cap A'^\perp\cap A'$,

(C2) $E\cap (A+A^\perp)\overset{\phi}{\approx}E'\cap (A'+A'^\perp)$.
\end{quote}
Now if $\phi:E\to E'$ can be extended to an isometry $\phi_{V}$ that sends $A$ to $A'$,
then (C1) and (C2) above and (C3) and (C4) below hold:
\begin{quote}
(C3) $E\cap A\overset{\phi}{\approx} E'\cap A'$,

(C4) $E\cap A^\perp\overset{\phi}{\approx} E'\cap A'^\perp$.
\end{quote}
Obviously, (C3) and (C4) imply (C1).

Suppose the isometry $\phi:E\to E'$ satisfies (C2)$\sim$(C4) (and so (C1)). Then $\phi$ induces two
linear bijections $\phi_A$ and $\phi_{A^\perp}$
of metric spaces as follow:
\ben
\item The first map is
\beq
\phi_A: \frac{(E+A^\perp)\cap A}{A^\perp\cap A}
\to \frac{(E'+A'^\perp)\cap A'}{A'^\perp\cap A'}.
\label{phi_A}
\eeq
The motivation of introducing $\phi_A$ originates from the special case $A^\perp\cap A=\{\0\}$,
where $V=A\odot A^\perp$ is nonsingular, and
$\frac{(E+A^\perp)\cap A}{A^\perp\cap A}\approx (E+A^\perp)\cap A$
is the projection of
$E$ onto $A$-component with respect to the decomposition $V=A\odot A^\perp$.
If $\phi$ can be extended to an isometry $\phi_{V}$ that sends $A$ to $A'$,
then the projection of $E$ onto \mbox{$A$-component} should be sent isometrically
to the projection of $E'$ onto \mbox{$A'$-component} in $V'=A'\odot A'^\perp$.

Let us define $\phi_A$ in \eqref{phi_A}.
Given $\v+(A^\perp\cap A)\in \frac{(E+A^\perp)\cap A}{A^\perp\cap A}$,
there exist $\e\in E$ and $\a_{\perp}\in A^\perp$
such that $\v=\e-\a_{\perp}\in A$, that is, $\e=\v+\a_{\perp}\in E\cap(A+A^\perp)$. Then
 $\phi(\e)\in E'\cap (A'+A'^\perp)$ by (C2).
So there exist $\v'\in A'$ and $\a'_{\perp}\in A'^\perp$
such that $\phi(\e)=\v'+\a'_{\perp}$, that is,
$\v'=\phi(\e)-\a'_{\perp}\in (E'+A'^\perp)\cap A'$. Define
\beq\phi_A\left(\v+(A^\perp\cap A)\right):=\v'+(A'^\perp\cap A').\eeq

The space $\frac{A+A^\perp}{A^\perp\cap A}$ (and so $\frac{(E+A^\perp)\cap A}{A^\perp\cap A}$)
carries
a metric  $\bar b$ induced from  $b$:
\beq\label{bar-b}
\bar b\left(\v_1+(A^\perp\cap A), \v_2+(A^\perp\cap A)\right)
:=b(\v_1, \v_2)
\qquad\text{for}\quad\v_1, \v_2\in A+A^\perp.
\eeq
Analogously, $\frac{A'+A'^\perp}{A'^\perp \cap A'}$
(and so $\frac{(E'+A'^\perp)\cap A'}{A'^\perp\cap A'}$)
carries a metric
$\bar b'$ induced from $b'$.

\begin{lem}\label{phi_A_well-defined}
 $\phi_A$ is a well-defined linear bijection of metric spaces.
\end{lem}

\begin{proof}
To establish the well-definedness of $\phi_{A}$, suppose
$\v_1+(A^\perp\cap A)  =\v_2+(A^\perp\cap A)$
in $\frac{(E+A^\perp)\cap A}{A^\perp\cap A}$,
where $\v_i=\e_i-{\a_{i\perp}}\in A$ for $\e_i\in E$ and $\a_{i\perp}\in A^\perp$,
 $i=1,2$. Then $\v_1-\v_2\in  A^\perp\cap A$ and so
$$(\e_1-\e_2)-(\a_{1\perp}-\a_{2\perp})
\in  A^\perp\cap A  .$$
Thus $\e_1-\e_2\in E\cap A^\perp$ by
$\a_{1\perp}-\a_{2\perp}\in A^\perp$, and
$\phi(\e_1)-\phi(\e_2)\in E'\cap A'^\perp$  by (C4).
Notice that $\phi(\e_i)\in E'\cap (A'+A'^\perp)$ by the preceding
discussion
in defining $\phi_A$.
Suppose $\phi(\e_i)=\v_i'+\a_{i\perp}'$ for $\v_i'\in A'$ and
$\a_{i\perp}'\in A'^\perp$, $i=1,2$. Then
\beq
\v_1'-\v_2'
=\left(\phi(\e_1)-\phi(\e_2)\right)-\left(\a_{1\perp}'-\a_{2\perp}'\right).
\eeq
Both $\phi(\e_1)-\phi(\e_2)$ and $\a_{1\perp}'-\a_{2\perp}'$ are in $A'^\perp$.
So $\v_1'-\v_2'\in A'^\perp\cap A'$;
that is,
$$\v_1'+(A'^\perp\cap A')=\v_2'+(A'^\perp\cap A').$$
Hence $\phi_A$ is well-defined.

The map $\phi_A$ is linear by its well-definedness and the linearity
of $\phi$,
and $\phi_A$ is a bijection since $\phi_{A}^{-1}$ can be defined
by $\phi^{-1}$ in a similar fashion.
\end{proof}

\item
Likewise, we define the second map
\beq
\phi_{A^\perp}: \frac{(E+A)\cap A^\perp}{A^\perp\cap A  }
\to \frac{(E'+A')\cap A'^\perp}{A'^\perp\cap A'  }.
\eeq
Given $\v+(A^\perp\cap A)\in \frac{(E+A)\cap A^\perp}{A^\perp\cap A  }$,
there are $\e\in E$ and $\a\in A$ such that
$\v=\e-\a\in A^\perp$. So $\e=\a+\v\in E\cap (A+A^\perp)$, and
thus $\phi(\e)\in E'\cap (A'+A'^\perp)$ by (C2).
There are $\a'\in A'$ and $\v'\in A'^\perp$
such that $\phi(\e)=\a'+\v'$. Define
\beq
\phi_{A^\perp}\left(\v+(A^\perp\cap A )\right)
:=\v'+(A'^\perp\cap A').
\eeq

\begin{lem}\label{phi_A^perp_well-defined}
$\phi_{A^\perp}$ is a well-defined linear bijection of metric spaces.
\end{lem}

The proof is analogous and so skipped.
\een

The following is one of the main theorems of this paper.

\begin{thm}\label{Witt-main}
Let $(V,b)$ and $(V',b')$ be isometric nonsingular symmetric metric spaces over $\F$.
Let subspaces $E, A\subset V$ and  $E', A'\subset V'$
satisfy that $A\approx A'$  and that $\phi:E\to E'$ is an isometry.
Then
$\phi$ can be extended to an isometry $\phi_{V}:V\to V'$ that
sends $A$ to $A'$  if and only if (C2)$\sim$(C4) hold and
one of $\phi_{A}$ and $\phi_{A^\perp}$ is an isometry.
\end{thm}

 Since $\phi$ is an isometry, one of $\phi_{A}$ and $\phi_{A^\perp}$ is
an isometry if and only if both of them are isometries
by the constructions of $\phi_{A}$ and $\phi_{A^\perp}$.

\begin{rem}
Theorem \ref{Witt-main} embraces several other results.
If $E\subset A$,
then $\phi_{A^\perp}$ is trivial,
and Theorem \ref{Witt-main} is  a combination of Lemma \ref{thm:Witt-deg}
and Witt's theorem.
If $E\perp A$, then $\phi_{A}$ is trivial, and
Theorem \ref{Witt-main} is
Corollary \ref{E-A-perp}.
If $A$ (resp. $A^\perp$) is totally isotropic,
then $\phi_{A}$ (resp. $\phi_{A^\perp}$) is trivial, and Theorem
\ref{Witt-main} is equivalent to Lemma \ref{thm:iso-Witt-gen}.
If $E\cap (A+A^\perp)=(E\cap A)+(E\cap A^\perp)$,
then Theorem \ref{Witt-main} becomes Corollary \ref{Witt-main-cor-1},
which
plays an important role in determining the isometry
of two generic flags in Section \ref{sect:flag}.
\end{rem}

\begin{proof}[Proof of Theorem \ref{Witt-main}]
It suffices to prove the sufficient part.
Recall that (C1) is implied by (C3) and (C4).

\ben
\item We shall extend $\phi|_{E\cap (A+A^\perp)}$ to an isometry
$\widehat\phi:A+A^\perp\to A'+A'^\perp$ that sends $A$ to $A'$,
 and
\beq\label{widehat-phi-sat}
b(\e,\v)
=b'(\phi(\e),\widehat\phi(\v))
\eeq
for all $\e\in E$ and $\v\in [E\cap(A+A^\perp)]+(A^\perp\cap A)$.

{\bf Step 1:} Decompose $E\cap (A+A^\perp)$.

Select $E_{A}$ and $E_{A^\perp}$
such that
\bea
E\cap A
&=& (E\cap A^\perp\cap A)\odot E_{A},
\label{E-cap-A}
\\
E\cap A^\perp
&=& (E\cap A^\perp\cap A)\odot E_{A^\perp}.
\label{E-cap-A^perp}
\eea
Then $E_{A}\cap (E\cap A^\perp)=
(E_{A}\cap A)\cap (E\cap A^\perp)=E_A\cap (E\cap A^\perp\cap A)
=\{\0\}$. So
\beq
(E\cap A)+(E\cap A^\perp)=(E\cap A^\perp\cap A)\odot E_{A}\odot E_{A^\perp}.
\label{E-cap-A-+-E-cap-A^perp}
\eeq
Select $E_{A, A^\perp}$ such that
\bea
E\cap (A+A^\perp)
&=& [(E\cap A)+(E\cap A^\perp)]\oplus E_{A, A^\perp}
\label{E-A-A^perp-complement}\\
&=& (E\cap A^\perp\cap A)\odot [(E_{A}\odot E_{A^\perp})\oplus E_{A, A^\perp}].
\label{E-cap-A-A^perp-decomp}
\eea

{\bf Step 2:} Decompose $A+A^\perp$ in accordance with \eqref{E-cap-A-A^perp-decomp}.

By the selection of $E_{A,A^\perp}$,
$$
E_{A,A^\perp}\cap A=\{\0\},\qquad
E_{A,A^\perp}\cap A^\perp=\{\0\},\qquad
E_{A,A^\perp}\subset A+A^\perp.
$$
Let $\{\v_1,\cdots,\v_k\}$ be a basis of $E_{A,A^\perp}$.
Then $\v_i=\a_i+\a_{i\perp}$ for some
$\a_i\in A$ and $\a_{i\perp}\in A^\perp$,
$i=1,\cdots, k$.
The following observations are useful:
\ben
\item
The choices of  $\a_i$ and $\a_{i\perp}$ for a given $\v_i$ are not unique
if $A^\perp\cap A\ne \{\0\}$.

\item The set $\{\a_1,\cdots,\a_k\}$ is always linearly independent.
Suppose $\sum_{i=1}^{k} t_i\a_i=\0$ for   $t_i\in\F$,
then $\sum_{i=1}^{k} t_i\v_i=\sum_{i=1}^{k} t_i\a_{i\perp}
\in E_{A,A^\perp}\cap A^\perp=\{\0\}$.
This forces $t_1=\cdots=t_k=0$ by
the linear independence of $\{\v_1,\cdots,\v_k\}$.
\item
Likewise,   $\{\a_{1\perp},\cdots,\a_{k\perp}\}$ is linearly independent.

\item
Denote
\bea\label{E_(A)-E_(A^perp)-def}
E_{(A)}
:= \spn\{\a_1,\cdots,\a_k\},
\qquad
E_{(A^\perp)}
:= \spn\{\a_{1\perp},\cdots,\a_{k\perp}\}.
\eea
Then $E_{(A)}\subset A$, $E_{(A^\perp)}\subset A^\perp$, and
\beq
E_{A,A^\perp}\subset E_{(A)}\odot E_{(A^\perp)}.
\label{E_A,A^perp}
\eeq
We claim that
\beq\label{E_(A)-E_(A^perp)}
\left ( E_{(A)}\odot E_{(A^{\perp})}\right )\cap \left [(E\cap A)+(E\cap A^\perp)+(A^\perp\cap A)\right]=\{\0\}.
\eeq
Suppose $\x=\a+\b_{\perp}=\e_A+\e_{A^\perp}+\a_{0}$
for $\a\in E_{(A)}$, $\b_{\perp}\in E_{(A^\perp)}$, $\e_A\in E\cap A$,
 $\e_{A^\perp}\in E\cap A^\perp$, and $\a_{0}\in A^\perp\cap A$,
 then by  \eqref{E_(A)-E_(A^perp)-def}, there exist
$\a_{\perp}\in E_{(A^\perp)}$ and $\e\in E_{A, A^\perp}$ such that
$\e=\a+\a_{\perp}$.
So
\beq\label{e-a-b-e-e}
\e=\a+\b_{\perp}+\a_{\perp}-\b_{\perp}=\e_{A}+\e_{A^\perp}+(\a_0+\a_{\perp}-\b_{\perp}).
\eeq
Then $\a_0+\a_{\perp}-\b_{\perp}\in A^\perp$
by $\a_0\in A^\perp\cap A$ and $\a_{\perp},\b_{\perp}\in E_{(A^\perp)}\subset A^\perp$.
Moreover, $\a_0+\a_{\perp}-\b_{\perp}\in E$ by \eqref{e-a-b-e-e} and
the fact
$\e, \e_{A}, \e_{A^\perp}\in E$.
So $\a_0+\a_{\perp}-\b_{\perp}\in E\cap A^\perp$. Thus
$\e\in (E\cap A)+(E\cap A^\perp)$
by \eqref{e-a-b-e-e}. This forces $\e=\0$ by $\e\in E_{A,A^\perp}$
and \eqref{E-A-A^perp-complement}. So $\a=\0$. Likewise, $\b_{\perp}=\0$.
Therefore, $\x=\0$ and  \eqref{E_(A)-E_(A^perp)} hold.

\item By \eqref{E-cap-A} and \eqref{E-cap-A^perp},
\bea
(E\cap A)+(A^\perp\cap A) &=& (A^\perp\cap A)\odot E_A,
\label{E-cap-A+A^perp-cap-A}
\\
(E\cap A^\perp)+(A^\perp\cap A) &=& (A^\perp\cap A)\odot E_{A^\perp}.
\label{E-cap-A^perp+A^perp-cap-A}
\eea
We claim that
\beq\label{E+A^perp-cap-A}
(E+A^\perp)\cap A=\left( [E\cap (A+A^\perp)]+A^\perp\right)\cap A.
\eeq
Every $\x\in (E+A^\perp)\cap A$ can be written as
$\x=\e-\a_{\perp}=\a$ for $\e\in E$, $\a_{\perp}\in A^\perp$, and $\a\in A$.
Then $\e=\a+\a_{\perp}\in E\cap(A+A^\perp)$ and so $\x\in \left([E\cap (A+A^\perp)]+A^\perp\right)\cap A$.
This proves \eqref{E+A^perp-cap-A}. By \eqref{E+A^perp-cap-A},
 \eqref{E-cap-A-A^perp-decomp}, and Lemma \ref{space}(5),
\bea
(E+A^\perp)\cap A
&=& [(E\cap A^\perp\cap A)+E_A+E_{A^\perp}+E_{A,A^\perp}+A^\perp]\cap A
\qquad
\notag
\\
&=& (E_A+E_{A,A^\perp}+A^\perp)\cap A
\notag
\\
&=& (E_A+ E_{(A)}+A^\perp)\cap A
\notag
\\
&=&
(A^\perp\cap A)\odot (E_{A}\oplus E_{(A)}).
\label{E+A^perp-cap-A-decomp}
\eea

Likewise,
\beq
(E+A)\cap A^\perp=(A^\perp\cap A)\odot (E_{A^\perp}\oplus E_{(A^\perp)}).
\label{E+A-cap-A^perp-decomp}
\eeq
By \eqref{E_(A)-E_(A^perp)},
\eqref{E+A^perp-cap-A-decomp}  and
\eqref{E+A-cap-A^perp-decomp},
\beq
(E+A^\perp)\cap A+(E+A)\cap A^\perp=
(A^\perp\cap A)\odot (E_{A}\oplus E_{(A)})\odot
(E_{A^\perp}\oplus E_{(A^\perp)}).
\label{0.48}
\eeq
By \eqref{E+A^perp-cap-A-decomp}, $E_A\oplus E_{(A)}$ can be extended to $V_A$ such that
\beq
A=(A^\perp\cap A)\odot V_A.
\label{A-decomp}
\eeq
By \eqref{E+A-cap-A^perp-decomp}, $E_{A^\perp}\oplus E_{(A^\perp)}$
can be extended to $V_{A^\perp}$ such that
\beq
A^\perp=(A^\perp\cap A)\odot V_{A^\perp}.
\label{A-perp-decomp}
\eeq
Then
\beq
A+A^\perp=(A^\perp\cap A)\odot V_{A}\odot V_{A^\perp}.
\label{V-+-V^perp-decomp}
\eeq
The decompositions \eqref{E-cap-A-A^perp-decomp} and \eqref{V-+-V^perp-decomp}
satisfy that
\beq
E\cap A^\perp\cap A\subset A^\perp\cap A,\quad
E_{A}\subset V_A,\quad
E_{A^\perp}\subset V_{A^\perp},\quad
E_{A,A^\perp}\subset V_{A}\odot V_{A^\perp}.
\eeq
Both $V_A$ and $V_{A^\perp}$ are nonsingular metric subspaces by \eqref{V-+-V^perp-decomp}.
\een

{\bf Step 3:}
Recall that $\frac{A+A^\perp}{A^\perp\cap A}$ has a  metric $\bar b$ defined by
\eqref{bar-b}.
The canonical quotient map $\rho: A+A^\perp\to \frac{A+A^\perp}{A^\perp\cap A}$
is metric-preserving.
By \eqref{E-cap-A+A^perp-cap-A},
\eqref{E+A^perp-cap-A-decomp},
\eqref{A-decomp},
\eqref{V-+-V^perp-decomp},
$\rho$ acts as isometries in the following commutative diagram
with inclusion maps:
\beq
\xymatrix{
E_{A} \ar[d]_{\rho} \ar@{^{(}->}[r]
    &E_{A}\oplus E_{(A)} \ar[d]_{\rho} \ar@{^{(}->}[r]
        &V_{A} \ar[d]_{\rho} \ar@{^{(}->}[r]
            &V_{A}\odot V_{A^\perp} \ar[d]_{\rho}
\\
\frac{(E\cap A)+(A^\perp\cap A)}{A^\perp\cap A} \ar@{^{(}->}[r]
    &\frac{(E+A^\perp)\cap A}{A^\perp\cap A} \ar@{^{(}->}[r]
        &\frac{A}{A^\perp\cap A} \ar@{^{(}->}[r]
            &\frac{A+A^\perp}{A^\perp\cap A}
}
\label{diagram-tau}
\eeq
Similarly, $\rho$ acts as isometries on the commutative diagram
obtained by interchanging  $A$ and $A^\perp$ in \eqref{diagram-tau}.

{\bf Step 4:} In $V'$,
denote $E_{A'}':=\phi(E_A)$, $E_{A'^\perp}':=\phi(E_{A^\perp})$,
$E_{A',A'^\perp}':=\phi(E_{A,A^\perp})$.
Then by (C1), (C3), and (C4),
\bea
E'\cap A'
&=& (E'\cap A'^\perp\cap A')\odot E'_{A'},
\\
E'\cap A'^\perp
&=& (E'\cap A'^\perp\cap A')\odot E'_{A'^\perp}.
\eea
By (C1) and (C4),
\beq
E'\cap (A'+A'^\perp)=
(E'\cap A'^\perp\cap A')\odot [(E'_{A'}\odot E'_{A'^\perp})\oplus E'_{A', A'^\perp}].
\label{0.56}
\eeq
The subspace
$E_{A',A'^\perp}'$ has a basis $\{\v_1',\cdots,\v_k'\}$
where $\v_i':=\phi(\v_i)$ for $i=1,\cdots,k$.
Write $\v_i'=\a_i'+\a_{i\perp}'$ for certain $\a_i'\in A'$ and $\a_{i\perp}'\in A'^\perp$.
Denote
\beq
E_{(A')}':=\spn\{\a_1',\cdots,\a_k'\},\qquad
E_{(A'^\perp)}':=\spn\{\a_{1\perp}',\cdots,\a_{k\perp}'\}.
\eeq
Extend $E'_{A'}\oplus E_{(A')}'$ to $V_{A'}'$ such that
\beq
A'=(A'^\perp\cap A')\odot V_{A'}.
\eeq
Extend
 $E'_{A'^\perp}\oplus E'_{(A'^\perp)}$ to $V'_{A'^\perp}$ such that
\beq
A'^\perp=(A'^\perp\cap A')\odot V'_{A'^\perp}.
\eeq
Then
\beq
A'+A'^\perp=(A'^\perp\cap A')\odot V'_{A'}\odot V'_{A'^\perp}.
\label{V'-+-V'^perp-decomp}
\eeq
The canonical quotient map
$\rho': A'+A'^\perp\to \frac{A'+A'^\perp}{A'^\perp\cap A'}$
is metric-preserving,
and $\rho'$ acts as isometries in the commutative diagram:
\beq
\xymatrix{
E'_{A'} \ar[d]_{\rho'} \ar@{^{(}->}[r]
    &E'_{A'}\oplus E'_{(A')} \ar[d]_{\rho'} \ar@{^{(}->}[r]
        &V'_{A'} \ar[d]_{\rho'} \ar@{^{(}->}[r]
            &V'_{A'}\odot V'_{A'^\perp} \ar[d]_{\rho'}
\\
\frac{(E'\cap A')+(A'^\perp\cap A')}{A'^\perp\cap A'} \ar@{^{(}->}[r]
    &\frac{(E'+A'^\perp)\cap A'}{A'^\perp\cap A'} \ar@{^{(}->}[r]
        &\frac{A'}{A'^\perp\cap A'} \ar@{^{(}->}[r]
            &\frac{A'+A'^\perp}{A'^\perp\cap A'}
}
\label{diagram-tau'}
\eeq
Likewise, $\rho'$ acts as isometries on the commutative diagram
obtained by interchanging  $A'$ and $A'^\perp$
in \eqref{diagram-tau'}.

{\bf Step 5:} We now extend $\phi|_{E\cap (A+A^\perp)}:E\cap (A+A^\perp)\to E'\cap (A'+A'^\perp)$
to an isometry $\widehat \phi:A+A^\perp\to A'+A'^\perp$,
such that $A\overset{\widehat\phi}{\approx}A'$,
 and
$b(\e,\v) =b'(\phi(\e),\widehat\phi(\v))$
 for all $\e\in E$ and $\v\in [E\cap (A+A^\perp)]+(A^\perp\cap A)$.

The isometry
$\rho'|_{E'_{A'}\oplus E'_{(A')}}: E'_{A'}\oplus E'_{(A')}\to
\frac{(E'+A'^\perp)\cap A'}{A'^\perp\cap A'}$
is invertible.
By the sufficient conditions in the theorem, $\phi_A$ defined in \eqref{phi_A} is an isometry. So
\beq
(\rho'|_{E'_{A'}\oplus E'_{(A')}})^{-1}\circ\phi_A\circ(\rho|_{E_A\oplus E_{(A)}}):
E_A\oplus E_{(A)}
\to E'_{A'}\oplus E'_{(A')}
\label{map-A}
\eeq
is an isometry.
Moreover, by the construction of $\phi_A$,
the map \eqref{map-A} restricted on  $E_A$
is equal to $\phi|_{E_A}$, and it
 sends $\a_i$ to $\a_i'$ for $i=1,\cdots,k$.
Apply  Witt's theorem on $E_A\oplus E_{(A)}\subset V_A$ and
$E'_{A'}\oplus E'_{(A')}\subset V'_{A'}$
where $V_A\approx \frac{A}{A^\perp\cap A}\approx\frac{A'}{A'^\perp\cap A'}\approx V'_{A'}$
are nonsingular.
The isometry \eqref{map-A} can be extended to
an isometry
\beq
\phi_{1}:V_{A}\to V'_{A'}.
\eeq
Likewise,
\beq
(\rho'|_{E'_{A'^\perp}\oplus E'_{(A'^\perp)}})^{-1}\circ\phi_{A^\perp}\circ(\rho|_{E_{A^\perp}\oplus E_{(A^\perp)}}):
E_{A^\perp}\oplus E_{(A^\perp)}
\to E'_{A'^\perp}\oplus E'_{(A'^\perp)}
\label{map-A^perp}
\eeq
is an isometry. Its restriction on $E_{A^\perp}$ is equal to
$\phi|_{E_{A^\perp}}$, and it sends
$\a_{i\perp}$ to $\a_{i\perp}'$ for $i=1,\cdots,k$.
By Witt's theorem, the map \eqref{map-A^perp} can be extended to an isometry
\beq
\phi_{2}:V_{A^\perp}\to V'_{A'^\perp}.
\eeq

Next, we extend $\phi|_{E\cap A^\perp\cap A}: E\cap A^\perp\cap A
\to E'\cap A'^\perp\cap A'$ to an isometry
$$\phi_0:
A^\perp\cap A\to A'^\perp\cap A'$$
such that $b(\e, \v)=b'(\phi(\e),\phi_0(\v))$
for all $\e\in E$ and $\v\in A^\perp\cap A$.
\ben
\item
Select $\widetilde E$ such that
\beq\label{widetilde-E}
E=[E\cap(A+A^\perp)]\oplus\widetilde E.
\eeq
Denote $\widetilde E':=\phi(\widetilde E)$. Then (C2) implies that
\beq\label{widetilde-E'}
E'=[E'\cap(A'+A'^\perp)]\oplus\widetilde E'.
\eeq
Let $\{\w_1,\cdots,\w_{\ell}\}$ be a basis of $E^\perp\cap E\cap A^\perp\cap A$.
On one hand, extend $\{\w_1,\cdots,\w_{\ell}\}$
to a basis $\{\w_1,\cdots,\w_{p}\}$ of $E\cap A^\perp\cap A$.
On the other hand, extend $\{\w_1,\dots,\w_{\ell}\}$ to a basis
$\{\w_1,\cdots,\w_{\ell}, \w_{p+1},\cdots,\w_{q}\}$ of $E^\perp\cap A^\perp\cap A$.
Then $\{\w_1,\cdots,\w_q\}$ is a basis of
$(E\cap A^\perp\cap A)+(E^\perp\cap A^\perp\cap A)$.
Extend $\{\w_1,\cdots,\w_q\}$ to a basis $\{\w_1,\cdots,\w_m\}$ of $A^\perp\cap A$. So $\ell\le p\le q\le m$.

\item
Since $V'$ is nonsingular, and
$(A'^\perp\cap A')+\widetilde E'^{\perp}=V'$ by taking orthogonal complement on
$(A'+A'^\perp)\cap\widetilde E'=\{\0\}$,
 the Riesz representation
$\tau': A'^\perp\cap A'\to (\widetilde E')^*,$
defined by
\beq\label{tau'}
\tau'(\v')(\e'):=b'(\e',\v')
\qquad \text{for}\quad \v'\in A'^\perp\cap A',\ \e'\in\widetilde E',
\eeq
is surjective.
Given one
$$
\w_i\in(A^\perp\cap A)-
\left[(E\cap A^\perp\cap A)+(E^\perp\cap A^\perp\cap A)\right]
\quad\text{for}\ \  i=q+1,\cdots,m,
$$  we define $\varphi_{\w_i}:\widetilde E'\to \F$ by
\beq
\varphi_{\w_i}(\e'):=b( \phi^{-1}(\e'), \w_i)
\qquad
\text{for \ }\e'\in\widetilde E'.
\eeq
Then $\varphi_{\w_i}\in (\widetilde E')^*$. There exists
$\w_i'\in A'^\perp\cap A'$ such that $\tau'(\w_i')=\varphi_{\w_i}$, that is,
\[
b'(\e',\w_i')=b(\phi^{-1}(\e'), \w_i)
\qquad\text{for \ }\e'\in\widetilde E'.
\]
Equivalently, $b'(\phi(\e),\w_i')=b(\e,\w_i)$
for $\e\in \widetilde E$ and $i=q+1,\cdots,m$.

\item By (C1) and Lemma \ref{isometry}(2),
  $\{\phi(\w_1),\cdots,\phi(\w_{\ell})\}$
is a basis of $E'^\perp\cap E'\cap A'^\perp\cap A'$. Then
\BEA
\dim (E^\perp\cap A^\perp\cap A)
&=& \dim V-\dim (E^\perp\cap A^\perp\cap A)^\perp
\\
&=&  \dim V-\dim (E+A+A^\perp)
\\
&=& \dim V-\dim E-\dim (A+A^\perp)+\dim [E\cap (A+A^\perp)].
\EEA
Notice that $A+A^\perp\approx A'+A'^\perp$ by Lemma \ref{isometry}.
So $\dim (E^\perp\cap A^\perp\cap A)=\dim (E'^\perp\cap A'^\perp\cap A')$.
Therefore,  $\{\phi(\w_1),\cdots,\phi(\w_{\ell})\}$ can be extended to
a basis $\{\phi(\w_1),\cdots,\phi(\w_{\ell}), \w_{p+1}',\cdots,\w_{q}'\}$
of $E'^\perp\cap A'^\perp\cap A'$.

\item
Define the linear map $\phi_0: A^\perp\cap A\to A'^\perp\cap A'$ by
\beq\label{phi_0}
\phi_0(\w_i):=
\bca
\phi(\w_i) &\text{for \ } i=1,\cdots,\ell,\ell+1,\cdots,p
\\
\w_i' &\text{for \ } i=p+1,\cdots, q, q+1,\cdots, m
\eca
\eeq
Then \eqref{phi_0} leads to the following conclusions:
\ben
\item[(1)]
The map $\phi_0$ is an extension of $\phi|_{E\cap A^\perp\cap A}$.
\item[(2)]
By (c), the restricted map $\phi_0|_{E^\perp\cap A^\perp\cap A}$ is a bijection.

\item[(3)]
$b'(\phi(\e), \phi_0(\w_i))=b'(\phi(\e), \phi(\w_i))=b(\e, \w_i)$
for $\e\in\widetilde E$ and
$i=1,\cdots,p$.

\item[(4)]
By (c),
$b'(\phi(\e), \phi_0(\w_i))
=b'(\phi(\e), \w_i')=0=b(\e, \w_i)$
for $\e\in\widetilde E$ and $i=p+1,\cdots,q$.

\item[(5)]
By  (b),
$b'(\phi(\e), \phi_0(\w_i))
=b'(\phi(\e), \w_i')=b(\e, \w_i)$
for $\e\in\widetilde E$ and $i=q+1,\cdots,m$.

\item[(6)] By  Lemma \ref{space}(3),
$b'( \phi(\e), \phi_0(\v)) =0=b(\e, \v) $ for
$\e\in E\cap(A+A^\perp)$ and $\v\in A^\perp\cap A$.
\een

Hence $b'(\phi(\e),\phi_0(\v)) =b( \e, \v) $ for
$\e\in E$ and $\v\in A^\perp\cap A$.
We claim that $\phi_0$ is an injection, so that
it is  an isometry between totally isotropic spaces
$A^\perp\cap A$ and $A'^\perp\cap A'$. If $\phi_0(\v)=\0'$
for $\v\in A^\perp\cap A$,
then $b( \e,\v) =b'(\phi(\e),\phi_0(\v)) =0$
for all $\e\in E$. Then $\v\in E^\perp\cap A^\perp\cap A$.
Then $\v=\0$
since $\phi_0|_{E^\perp\cap A^\perp\cap A}$
is a bijection by (2).
So $\phi_0$ is an injection.
Therefore,
$\phi_0$ in \eqref{phi_0}
is   an isometry  that extends $\phi|_{E\cap A^\perp\cap A}$,
and $b'(\phi(\e),\phi_0(\v)) =b(\e, \v) $
for $\e\in E$ and $\v\in A^\perp\cap A$.

\een

By \eqref{V-+-V^perp-decomp} and \eqref{V'-+-V'^perp-decomp},
the map
\beq
\widehat \phi:=\phi_0\odot\phi_1\odot \phi_2
\eeq
is an isometry
from $A+A^\perp$ to $A'+A'^\perp$, where
\BEQ
\widehat\phi(A)
=
\phi_0(A^\perp\cap A)\odot\phi_1(V_A)=
(A'^\perp\cap A')\odot V'_{A'}=A'.
\EEQ
By \eqref{E-cap-A-A^perp-decomp}, \eqref{0.56},
and the above discussion,
$\widehat \phi$ is an extension of $\phi|_{E\cap(A+A^\perp)}$.
Clearly $b'(\phi(\e),\widehat\phi(\v)) =b(\e,\v) $
for $\e\in E$ and $\v\in [E\cap (A+A^\perp)]+(A^\perp\cap A)$.
This is \eqref{widehat-phi-sat}.

\item
Now the isometries $\widehat \phi: A+A^\perp\to A'+A'^\perp$ and
$\phi: E\to E'$ agree on
$\Dom(\widehat\phi)\cap\Dom(\phi)=E\cap (A+ A^\perp).$
However, the combination of these two isometries may not be an isometry
--- given  $\widetilde \e\in\widetilde E$
and
$\v\in (A+A^\perp)-\left([E\cap (A+A^\perp)]+(A^\perp\cap A)\right)$, it is not guaranteed that
$b(\widetilde\e,\v)=b'(\phi(\e), \widehat\phi(\v))$.
So $\widehat\phi$ should be adjusted slightly.

Choose a basis $\{\u_1,\cdots,\u_{r}\}$ of $[E\cap(A+A^\perp)]+(A^\perp\cap A)$.
Then
 extend it to a basis $\{\u_1,\cdots,\u_s\}$ of $A+A^\perp$.
Given one $\u_i$ for $i=r+1,\cdots, s$, we define
$\mu_{\u_i}:\widetilde E'\to\F$ by
$$
\mu_{\u_i}(\e'):=b'(\e',\widehat\phi(\u_i)) -b(\phi^{-1}(\e'),\u_i)
\qquad\text{for}\quad\e'\in\widetilde E'.
$$
Then $\mu_{\u_i}\in (\widetilde E')^*$.
By the surjectivity of $\tau'$ defined in \eqref{tau'},
we can select  $\v'_{\u_i}\in A'^\perp\cap A'$
such that
$$
b'(\e',\widehat\phi(\u_i)) -b(\phi^{-1}(\e'),\u_i)
=b'( \e', \v'_{\u_i})
\qquad\text{for}\quad \e' \in\widetilde E'.
$$
Let $\e':=\phi(\e)$ for $\e\in\widetilde E$. Then
\beq
b'(\phi(\e),\widehat\phi(\u_i)-\v'_{\u_i}) =b(\e,\u_i)
\qquad\text{for}\quad\e\in\widetilde E.
\eeq
Define the linear map $\widetilde\phi:E+A+A^\perp\to E'+A'+A'^\perp$  by
\beq\label{widetilde-phi}
\bca
\widetilde\phi(\u_i)
:= \widehat\phi(\u_i), &\text{for \ } i=1,\cdots,r
\\
\widetilde\phi(\u_i)
:= \widehat\phi(\u_i)-\v'_{\u_i}, &\text{for \ } i=r+1,\cdots,s
\\
\widetilde\phi(\e)
:= \phi(\e), &\text{for \ } \e\in\widetilde E
\eca
\eeq
The construction \eqref{widetilde-phi} leads to the following conclusions:
\ben
\item[(1)] $\widetilde\phi$ is an extension of $\phi$
since $\widetilde\phi|_{E}=\widehat\phi|_{E\cap(A+A^\perp)}\oplus\phi|_{\widetilde E}=\phi$.
\item[(2)] $\widetilde\phi(A)\subset \widehat\phi(A)+(A'^\perp\cap A')=A'$.
\item[(3)] $\widetilde\phi$ is metric-preserving.

\item[(4)]
We claim that $\widetilde\phi$ is a bijection. By \eqref{widetilde-E} and
\eqref{widetilde-E'},
\beq
E+A+A^\perp=(A+A^\perp)\oplus \widetilde E,
\qquad
E'+A'+A'^\perp=(A'+A'^\perp)\oplus \widetilde E',
\eeq
Choose a basis $\{\widetilde \e_1,\cdots,\widetilde \e_t\}$ of $\widetilde E$.
Then $\{\u_1,\cdots,\u_s,\widetilde\e_1,\cdots,\widetilde\e_t\}$
is a basis of $E+A+A^\perp$, and
$\{\widehat\phi(\u_1),\cdots,\widehat\phi(\u_s),\phi(\widetilde\e_1),\cdots,
\phi(\widetilde\e_t)\}$ is a basis of $E'+A'+A'^\perp$.
By \eqref{widetilde-phi}, the matrix representation of $\widetilde\phi$
with respect to the above two bases is a unit upper triangular matrix,
which is nonsingular.
So the map $\widetilde\phi$ is a bijection.
\een
\een

Therefore, $\widetilde\phi$ defined in \eqref{widetilde-phi}
is an isometry that extends $\phi$, and $\widetilde\phi(A)=A'$.
By Witt's theorem, $\widetilde\phi$ can be extended to
an isometry $\phi_{V}:V\to V'$.
Clearly the isometry $\phi_{V}$ is an extension of $\phi:E\to E'$ that sends
$A$ to $A'$. We are done.
\end{proof}

By \eqref{E-cap-A},
\eqref{E-cap-A^perp},
\eqref{E-cap-A-A^perp-decomp},
\eqref{E+A^perp-cap-A-decomp},
and
\eqref{E+A-cap-A^perp-decomp},
the following three equalities about $E, A\subset V$ are equivalent:
\bse
\bea
(E+A^\perp)\cap A &=& E\cap A,
\\
(E+A)\cap A^\perp &=& E\cap A^\perp,
\\
E\cap (A+A^\perp) &=& (E\cap A)+(E\cap A^\perp).
\eea
\ese
Theorem \ref{Witt-main} implies the following result:

\begin{cor}\label{Witt-main-cor-1}
Let $E, A\subset V$ and $E', A'\subset V'$ satisfy that
$E\approx E'$, $A\approx A'$,
$(E+A^\perp)\cap A=E\cap A$ and $(E'+A'^\perp)\cap A'=E'\cap A'$.
Then an isometry $\phi:E\to E'$ can be extended to an isometry
$\phi_{V}:V\to V'$ that also sends $A$ to $A'$  if and only if
(C3) and (C4) hold.
\end{cor}

\begin{proof}
It suffices to prove the sufficient part. Suppose that
(C3): $E\cap A\overset{\phi}\approx E'\cap A'$ and
(C4): $E\cap A^\perp\overset{\phi}\approx E'\cap A'^\perp$ hold.
By the sufficient conditions in the corollary and the preceding discussion,
\BEA
E\cap(A+A^\perp)
&=& (E\cap A)+(E\cap A^\perp)
\\
&\overset{\phi}\approx &
(E'\cap A')+(E'\cap A'^\perp)
=E'\cap (A'+A'^\perp).
\EEA
So (C2) holds. Moreover, the induced linear bijection $\phi_{A}$ in
\eqref{phi_A} is given by
$$
\phi_{A}\left(\v+(A^\perp\cap A)\right)=\phi(\v)+(A^\perp\cap A),
\qquad\text{for}\quad \v\in E\cap A,
$$
which is   an isometry.
So  Corollary \ref{Witt-main-cor-1} follows by Theorem \ref{Witt-main}.
\end{proof}

\begin{rem}\label{counter-example}
If  an isometry $\phi_{V}:V\to V'$  sends $E$ to $E'$ and $A$ to $A'$
respectively,
then
\beq
f(E,A)\approx \phi_{V}(f(E, A))=f(\phi_{V}(E),\phi_{V}(A))=f(E',A')
\eeq
for every meaningful expression $f$
that consists of brackets, $+$, $\cap$, and $\perp$ (taking orthogonal complement).
Nevertheless,  the converse is not true!
In the following counterexample,  $E, A\subset V$ and $E', A'\subset V'$ satisfy that
 $f(E,A)\approx f(E',A')$
for every   $f$ defined above.
However, there is no isometry $\phi_{V}:V\to V'$
that sends $E$ to $E'$ and $A$ to $A'$ respectively.

Suppose that $\text{char}(\F)=0$. Let $V=V'=\spn\{\v_1,\v_2\}$ be the
nonsingular symmetric metric space
determined by $b(\v_i,\v_j) =\delta_{ij}$ for $i,j=1,2$.
 Denote
\[
A=A':=\F\v_1,\qquad
E:=\F(\v_1+2\v_2),\qquad
E':=\F(\v_1+3\v_2).
\]
Then
\[
A^\perp=A'^\perp=\F\v_2,\qquad
E^\perp=\F(2\v_1-\v_2),\qquad
E'^\perp=\F(3\v_1-\v_2).
\]
Denote  a set ${\mc S}$ (resp. ${\mc S}'$) of subspaces of $V$ (resp. $V'$) as follow:
\[
{\mc S}:=\{\{\0\}, E, E^\perp, A, A^\perp, V\},
\qquad
{\mc S}':=\{\{\0'\}, E', E'^\perp, A', A'^\perp, V'\}.
\]
It is easy to check that ${\mc S}$ (resp. ${\mc S}'$) is closed under
the operations $+$, $\cap$, and $\perp$.
Moreover, define a bijection $\varpi:{\mc S}\to {\mc S}'$ that
sends $\{\0\}$ to $\{\0'\}$,
$E$ to $E'$,
$E^\perp$ to $E'^\perp$,
$A$ to $A'$,
$A^\perp$ to $A'^\perp$,
$V$ to $V'$, respectively.
Then $X\approx\varpi(X)$ for every $X\in {\mc S}$, and
 $\varpi$ commutes with $+$, $\cap$, and $\perp$ operations.
For every meaningful expression $f(X,Y)$ of subspaces $X$ and $Y$ in a metric space that consists of
brackets, $+$, $\cap$, and $\perp$, we have $f(E,A)\in {\mc S}$ and
\[
f(E,A)\approx \varpi(f(E,A))=f(\varpi(E), \varpi(A))=f(E',A').
\]
However, there is no isometry $\phi_{V}:V\to V'$ that sends
$A$ to $A'$ and $E$ to $E'$ simultaneously.
\end{rem}

\begin{rem}
The equality (C1) is implied by (C3) and (C4). However,
if the specific isometry $\phi$ is removed from (C1)$\sim$(C4),
then the resulting relations are independent of each other.

To take a glance, let $V=V'=\spn\{\e_1, \widetilde\e_1, \e_2, \widetilde\e_2, \e_3, \widetilde\e_3\}$
be the nonsingular symmetric metric space defined by $b=b'$ and
\[
b( \e_i, \e_j) =0,\qquad
b( \widetilde\e_i, \widetilde\e_j) =0,\qquad
b( \e_i, \widetilde\e_j) =\delta_{ij},\qquad
\text{for \ } i, j=1, 2, 3.
\]
Denote
\[
E=E' := \spn\{\e_1, \e_2+\e_3\},
\quad
A := \spn\{\e_1, \e_2, \widetilde\e_2\},
\quad
A' := \spn\{\e_1, \widetilde \e_1, \widetilde \e_2-\widetilde\e_3\}.
\]
Then
\[
A^\perp = \spn\{\e_1, \e_3, \widetilde\e_3\},\qquad
A'^\perp = \spn\{\e_2+\e_3, \widetilde\e_2, \widetilde\e_3\}.
\]
Evidently,
$E\approx E'$, $A\approx A'$, $E\cap (A+A^\perp)\approx E'\cap(A'+A'^\perp)$,
$E\cap A\approx E'\cap A'$, $E\cap A^\perp\approx E'\cap A'^\perp$,
but $E\cap A^\perp\cap A$ and $E'\cap A'^\perp\cap A'$ are not isometric.
\end{rem}

\begin{rem}\label{phi_A-phi-A^perp}
Suppose $E, A\subset V$, $E', A'\subset V'$, and $\phi:E\to E'$ satisfy
(C2)$\sim$(C4), so that $\phi_{A}$ and $\phi_{A^\perp}$
are well-defined.
In the metric space $(\frac{A+A^\perp}{A^\perp\cap A},\bar b)$, the intersection of
$\Dom(\phi_{A})=\frac{(E+A^\perp)\cap A}{A^\perp\cap A}$ and
$\Dom(\phi_{A^\perp})=\frac{(E+A)\cap A^\perp}{A^\perp\cap A}$ is trivial,
and $\Dom(\phi_{A})\perp\Dom(\phi_{A^\perp})$.
Likewise for $\Ima(\phi_{A})$ and $\Ima(\phi_{A^\perp})$
in   $(\frac{V'}{A'^\perp\cap A'},\bar b')$.
So we can define the linear bijection
\beq
\phi_A\odot\phi_{A^\perp}:
\frac{[(E+A^\perp)\cap A]+[(E+A)\cap A^\perp]}{A^\perp\cap A}
\to \frac{[(E'+A'^\perp)\cap A']+[(E'+A')\cap A'^\perp]}{A'^\perp\cap A'}.
\eeq
Then
$\phi_A\odot\phi_{A^\perp}$ is an extension of the following isometry
induced by $\phi$:
\begin{equation}
\phi_{A+A^\perp}:\frac{[E\cap (A+A^\perp)]+(A^\perp\cap A)}{A^\perp\cap A }
\to
\frac{[E'\cap (A'+A'^\perp)]+(A'^\perp\cap A') }{A'^\perp\cap A' },
\end{equation}
where
\ $\phi_{A+A^\perp}\left(\e+(A^\perp\cap A)\right):=\phi(\e)+(A'^\perp\cap A')$ \
for \ $\e\in E\cap (A+A^\perp)$.
\end{rem}

\section{More about isometries over flags}\label{sect:flag}

We continue the study on isometry of flags in
 Section \ref{sect:self-dual-I}.
First we give a criteria about the isometry of two generic flags.
Then we determine when there exists an isometry that
simultaneously sends a subspace to another, and
a self-dual flag to another, respectively.
The later object is analogous to Theorem \ref{thm:iso-Witt-ext}
except that here   no isometry $\phi$ is provided.

Let $V\approx V'$ be nonsingular isometric metric spaces.
The following theorem is an extension of Lemma \ref{E-V-*}.

\begin{thm}\label{Witt-flag}
Let ${\mc V}=\{V_i\}_{i=0,\cdots,k}$
and ${\mc V}'=\{V_i'\}_{i=0,\cdots,k}$ be generic flags of
$V$ and $V'$ respectively.
Then ${\mc V}\approx {\mc V}'$ if and only if
$V_i\approx V_i'$ for $i=0,\cdots,k$ and
$\dim (V_i^\perp\cap V_j)=\dim (V_i'^\perp\cap V_j')$ for
$0<j<i<k$.
\end{thm}

\begin{proof}
We only need to prove the sufficient part.

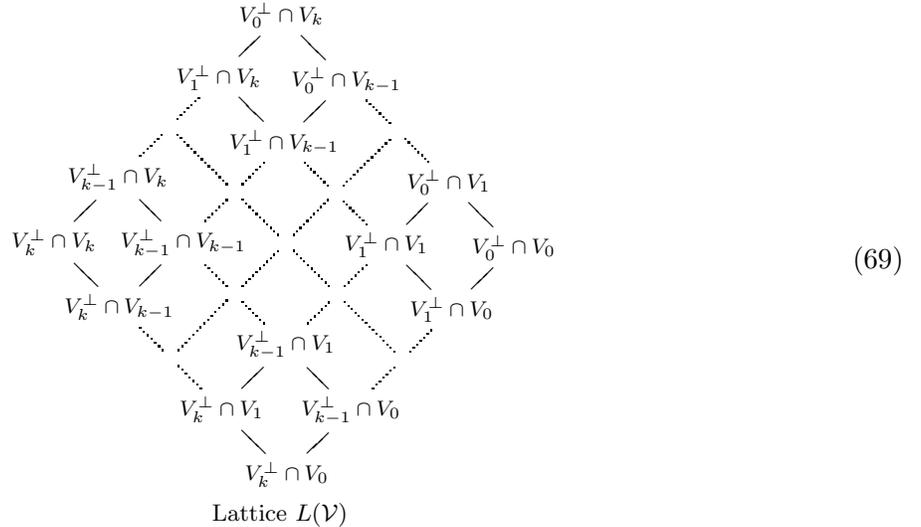
\beq
\begin{tabular}{c}
\scriptsize{
\xymatrix @dr @R=4mm @C=4mm {
V_{0}^\perp\cap V_{k}\ar@{-}[d]\ar@{-}[r]
    &V_{0}^\perp\cap V_{k-1}\ar@{-}[d]\ar@{.}[r]
        &\ar@{.}[d]\ar@{.}[r]
            &V_{0}^\perp\cap V_{1}\ar@{-}[d]\ar@{-}[r]
                &V_{0}^\perp\cap V_{0}\ar@{-}[d]
\\
V_{1}^\perp\cap V_{k}\ar@{.}[d]\ar@{-}[r]
    &V_{1}^\perp\cap V_{k-1}\ar@{.}[d]\ar@{.}[r]
        &\ar@{.}[d]\ar@{.}[r]
            &V_{1}^\perp\cap V_{1}\ar@{.}[d]\ar@{-}[r]
                &V_{1}^\perp\cap V_{0}\ar@{.}[d]
\\
\ar@{.}[d]\ar@{.}[r]
    &\ar@{.}[d]\ar@{.}[r]
        &\ar@{.}[d]\ar@{.}[r]
            &\ar@{.}[d]\ar@{.}[r]
                &\ar@{.}[d]
\\
V_{k-1}^\perp\cap V_{k}\ar@{-}[d]\ar@{-}[r]
    &V_{k-1}^\perp\cap V_{k-1}\ar@{-}[d]\ar@{.}[r]
        &\ar@{.}[d]\ar@{.}[r]
            &V_{k-1}^\perp\cap V_{1}\ar@{-}[d]\ar@{-}[r]
                &V_{k-1}^\perp\cap V_{0}\ar@{-}[d]
\\
V_{k}^\perp\cap V_{k}\ar@{-}[r]
    &V_{k}^\perp\cap V_{k-1}\ar@{.}[r]
        &\ar@{.}[r]
            &V_{k}^\perp\cap V_{1}\ar@{-}[r]
                &V_{k}^\perp\cap V_{0}
}
}
\\
\text{\footnotesize Lattice $L({\mc V})$}
\end{tabular}
\eeq
The figure above displays a lattice $L({\mc V})$ of subspaces in $V$, where
the edges represent inclusion relations.
The entries of $L({\mc V})$ are $V_i^\perp\cap V_j$ where $(i,j)$ is in the
index set
\beq
I:=\{0,\cdots,k\}\times\{0,\cdots,k\}.
\eeq
Let us call $V_i^\perp\cap V_j$ the $(i,j)$-entry of $L({\mc V})$.
Of course, $V_k^\perp\cap V_j=V_i^\perp\cap V_0=\{\0\}$,
$V_0^\perp\cap V_j=V_j$ and $V_i^\perp\cap V_k=V_i^\perp$.
Likewise, we denote a lattice $L({\mc V}')$ with respect to ${\mc V}'$.

The subspaces in the ``lower half part'' of $L({\mc V})$ (including those in the
middle) are $V_i^\perp\cap V_j$ for $0\le j\le i\le k$,
which are totally isotropic.
Moreover, denote
\beq
T:=\sum_{i=0}^{k} (V_i^\perp\cap V_i),\qquad
T':=\sum_{i=0}^{k} (V_i'^\perp\cap V_i').
\eeq
Then both $T$ and $T'$ are totally isotropic.
By the sufficient conditions in the theorem, $\dim(V_i^\perp\cap V_j)=\dim(V_i'^\perp\cap V_j')$
for $0\le j<i\le k$,
and $\dim(V_i^\perp\cap V_i)=\dim(V_i'^\perp\cap V_i')$
since $V_i\approx V_i'$.
Using basis extension and induction over the lattice entries
of $L({\mc V})$ and $L({\mc V}')$,
we can  explicitly construct a linear bijection
$\phi_0:T\to T'$ such that $\phi_0(V_i^\perp\cap V_j)=V_i'^\perp\cap V_j'$
for $0\le j\le i\le k$. Then $\phi_0$ is an isometry
of totally isotropic subspaces.

Given $j\in\{0,\cdots,k-1\}$, suppose
$\phi_j:T+V_j\to T'+V_j'$ is an isometry such that
$\phi_j|_{T}=\phi_0$ and $\phi_j(V_q)=V_q$ for $q=0,\cdots,j$.
We show that $\phi_j$ can be extended to an isometry
$\phi_{j+1}:T+V_{j+1}\to T'+V_{j+1}'$ such that
 $\phi_{j+1}(V_{j+1})=V_{j+1}'$. This is done by
applying Corollary \ref{Witt-main-cor-1} on $(V,b)\approx (V',b')$ with
\beq
E:=T+V_j,\quad
A:=V_{j+1},\quad E':=T'+V_j',\quad A':=V_{j+1}',\quad\text{and}\quad
\phi:=\phi_j.
\eeq
Let us  verify the sufficient conditions in Corollary \ref{Witt-main-cor-1}:
\ben
\item[(C3):]
\BEA
E\cap A &=& \left [\sum_{i=0}^k(V_i^\perp\cap V_i)+V_j\right]\cap V_{j+1}
= \left [\sum_{i=j+1}^k(V_i^\perp\cap V_i)+V_j\right]\cap V_{j+1}
\\
&=& \left(\left [\sum_{i=j+1}^k(V_i^\perp\cap V_i)\right]\cap V_{j+1}\right)+V_j
\qquad \text{by Lemma \ref{space}(5)}
\\
&=& \left(\left [\sum_{i=j+1}^k(V_i^\perp\cap V_i)\right]\cap V_{j+1}^\perp\cap V_{j+1}\right)+V_j
\qquad \text{by Lemma \ref{space}(6)}
\\
&=& (V_{j+1}^\perp\cap V_{j+1})+V_j.
\EEA
Likewise, $E'\cap A'=(V_{j+1}'^\perp\cap V_{j+1}')+V_j'$. Then (C3) holds since
$$
\phi\left((V_{j+1}^\perp\cap V_{j+1})+V_j\right)=\phi_0(V_{j+1}^\perp\cap V_{j+1})
+\phi_j(V_j)=(V_{j+1}'^\perp\cap V_{j+1}')+V_j'.
$$

\item[(C4):]
\BEA
E\cap A^\perp &=&
\left [\sum_{i=0}^k(V_i^\perp\cap V_i)+V_j\right]\cap V_{j+1}^\perp
= \left [\sum_{i=j+1}^k(V_i^\perp\cap V_i)+V_j\right]\cap V_{j+1}^\perp
\\
&=&
\left [\sum_{i=j+1}^k(V_i^\perp\cap V_i)\right]+\left(V_{j+1}^\perp\cap V_j\right)
\qquad\text{by Lemma \ref{space}(5)}
\\
&=&
\sum_{i=j+1}^k(V_i^\perp\cap V_i).
\EEA
Then $E\cap A^\perp\overset{\phi}{\approx}E'\cap A'^\perp$
by similar computation on $E'\cap A'^\perp$.
Thus (C4) holds.
\een

Besides these,
\BEA
(E+A^\perp)\cap A &=&
\left [\sum_{i=0}^k(V_i^\perp\cap V_i)+V_j+V_{j+1}^\perp\right]\cap V_{j+1}
\\
&=& \left [\sum_{i=j+1}^k(V_i^\perp\cap V_i)+V_j+V_{j+1}^\perp\right]\cap V_{j+1}
\\
&=&  \left (V_j+V_{j+1}^\perp\right )\cap V_{j+1}
\\
&=& V_j+(V_{j+1}^\perp\cap V_{j+1})\qquad\text{by Lemma \ref{space}(5)}
\\
&=& E\cap A.
\EEA
Similarly, $(E'+A'^\perp)\cap A'=E'\cap A'$.
By Corollary \ref{Witt-main-cor-1},
the isometry $\phi=\phi_j$ can be extended to an isometry
$\widetilde\phi_j:V\to V'$ that sends $A=V_{j+1}$ to $A'=V_{j+1}'$.
Then the isometry $\phi_{j+1}:=\widetilde\phi_j|_{T+V_{j+1}}$
is an extension of $\phi_j$ that sends $V_{j+1}$ to $V_{j+1}'$.

Finally,
repeating the above process,
we obtain a series of extension of isometries via
$\phi_0\to \phi_1\to \cdots\to\phi_k$.
Notice that $T+V_k=V$ and $T'+V_k'=V'$.
The isometry $\phi_k:V\to V'$ satisfies that
$\phi_{k}(V_i)=V_i'$ for $i=0,\cdots,k$.
So ${\mc V}\approx {\mc V}'$.
\end{proof}

We return to the simultaneous isometry of (subspace, self-dual flag) pairs.
As before, denote $E_i:=E\cap V_i$ and
$E_i':=E'\cap V_i'$.
Recall that two self-dual flags
${\mc V} =\{V_i\}_{i=0,\cdots,k}$ of $V$ and
${\mc V}' =\{V_i'\}_{i=0,\cdots,k}$ of $V'$
are isometric if and only if
$\dim V_i=\dim V_i'$ for $i=1,\cdots,\lfloor \frac k2\rfloor$.
The following  result is analogous to Theorem \ref{thm:iso-Witt-ext}.

\begin{thm}\label{thm:iso-Witt-like}
Suppose $E\subset V$ and $E'\subset V'$ where $V\approx V'$.
Let ${\mc V}:=\{V_i\}_{i=0,\cdots,k}$ and
${\mc V}':=\{V_i'\}_{i=0,\cdots,k}$ be isometric self-dual flags
 of $V$ and $V'$ respectively.
Then there exists an isometry $\phi_{V}:V\to V'$ that sends
$E$ to $E'$ and ${\mc V}$ to ${\mc V}'$ simultaneously
if and only if $E_i\approx  E_i'$ for $i=1,\cdots,k$, and
\ $\dim (E_i^\perp\cap E_j)= \dim (E_i'^\perp\cap E_j')$ \ for \
$k-i< j< i\le k$.
\end{thm}

\begin{proof}
It suffices to prove the sufficient part.
Note that $E_i^\perp\cap E_j=E_j$ whenever $i+j\le k$.
We have $\dim(E_i^\perp\cap E_j)=\dim (E_i'^\perp\cap E_j')$
for $0<j<i\le k$. Consider the flags
$${\mc E}:=\{E_0=\{\0\}\subset E_1\subset\cdots\subset E_k=E\subset V\}
\qquad\text{of}\qquad V$$
and
$${\mc E}':=\{E_0'=\{\0'\}\subset E_1'\subset\cdots\subset E_k'=E'\subset V'\}
\qquad\text{of}\qquad V'.$$
By Theorem \ref{Witt-flag},
there is an isometry $\phi:V\to V'$ that sends
${\mc E}$ to ${\mc E}'$, that is,
$\phi(E_i)=E_i'$ for $i=0,\cdots,k$.
By Theorem \ref{thm:iso-Witt-ext},
the isometry $\phi|_E:E\to E'$ can be extended to an isometry
$\phi_{V}:V\to V'$ that sends ${\mc V}$ to ${\mc V}'$.
This completes the proof.
\end{proof}

\begin{cor}\label{E-A-isotropic}
Suppose $V\approx V'$.
Suppose $E, A\subset V$ and $E', A'\subset V'$
where $A\approx A'$ are totally isotropic and $E\approx E'$.
Then there exists an isometry $\phi_V:V\to V'$ that sends
$E$ to $E'$ and $A$ to $A'$ simultaneously
if and only if
\bse
\bea
E\cap A^\perp &\approx& E'\cap A'^\perp,\\
\dim (E\cap A) &=& \dim (E'\cap A'),\\
\dim(E^\perp\cap E\cap A) &=& \dim (E'^\perp\cap E'\cap A'),\\
\dim(E^\perp\cap E\cap A^\perp) &=& \dim (E'^\perp\cap E'\cap A'^\perp).
\eea
\ese
\end{cor}

\begin{proof}
Denote the self-dual flags ${\mc V}:=\{\{\0\}\subset A\subset A^\perp\subset V\}$
and ${\mc V}':=\{\{\0'\}\subset A'\subset A'^\perp\subset V'\}$.
Then  apply Theorem \ref{thm:iso-Witt-like} to get the result.
\end{proof}

If each of $E$, $A$, $E^\perp$, and $A^\perp$
in Corollary \ref{E-A-isotropic} is not totally
isotropic, then by Remark \ref{counter-example}, we may not be able to
determine the simultaneous isometry  of $(E,A)$ pair with $(E', A')$ pair
by solely inspecting the isometries of the
corresponding subspaces related to the pairs
via intersections, additions, and taking orthogonal complements.


\end{document}